\documentclass[preprint,12pt]{elsarticle}

\usepackage[a4paper,margin=2.5cm]{geometry}
\usepackage{amssymb,amsmath,amsfonts,latexsym}
\usepackage{mathtools}
\usepackage{multirow,array}
\usepackage{graphicx}
\usepackage{float}
\usepackage{booktabs}
\usepackage{xcolor}
\usepackage{changepage}
\usepackage{pgfplots}
\pgfplotsset{compat=1.15}
\usepgfplotslibrary{groupplots}
\usepackage[hidelinks]{hyperref}

\numberwithin{equation}{section}

\definecolor{serA}{HTML}{2A78D6}
\definecolor{serB}{HTML}{EB6834}
\definecolor{serC}{HTML}{1BAF7A}
\newcommand{\Cnorm}[1]{\|#1\|_{1,\infty}}
\newcommand{\IR}{\mathcal{I}^{(\beta)}_{[0,L]}}
\def\RR{{{\mathbb R}}}

\newtheorem{theorem}{Theorem}
\newtheorem{lemma}{Lemma}
\newtheorem{proposition}{Proposition}
\newtheorem{corollary}{Corollary}
\newtheorem{remark}{Remark}
\newtheorem{hypothesis}{Hypothesis}

\def\proofname{\smallskip \noindent{\it Proof}. }
\def\proof{\par\proofname}
\def\proofend{\hfill$\Box$\par\medskip}

\allowdisplaybreaks
\journal{Journal of Computational and Applied Mathematics}

\begin{document}

\begin{frontmatter}

\title{Analytical and numerical study of nonlinear Riesz--Caputo boundary value problems}

\author[aq]{Enza Pellegrino}
\ead{enza.pellegrino@univaq.it}

\author[sbai]{Francesca Pitolli}
\ead{francesca.pitolli@uniroma1.it}

\author[ingv]{Chiara Sorgentone\corref{cor1}}
\ead{chiara.sorgentone@ingv.it}

\cortext[cor1]{Corresponding author.}
\fntext[pa]{Present address: Istituto Nazionale di Geofisica e Vulcanologia (INGV), Via di Vigna Murata 605, 00143 Rome, Italy.}

\address[aq]{Department of Industrial and Information Engineering and Economics, University of L'Aquila, Roio Poggio, 67040 L'Aquila, Italy}

\address[sbai]{Department of Basic and Applied Sciences for Engineering, Sapienza University of Rome, Via Antonio Scarpa 16, 00161 Rome, Italy}

\address[ingv]{Istituto Nazionale di Geofisica e Vulcanologia (INGV), Via di Vigna Murata 605, 00143 Rome, Italy.}

\begin{abstract}
Fractional boundary value problems are often used to model complex systems and processes characterized by memory effects and anomalous diffusion. In this paper, we consider nonlinear fractional boundary value problems of arbitrary order $\alpha$ involving the Riesz--Caputo operator, which is particularly suited for modelling physical phenomena exhibiting symmetric diffusive effects. By reducing the problem to a dominant singular integral equation of Cauchy type and inverting it in closed form, we derive an integral representation of the solution, which we use to prove existence and uniqueness. The closed-form inversion also describes the typical behaviour of the solution near the endpoints of the interval and provides an exact singular solution, which we use as a test case. We then analyse a B-spline collocation method for the approximation of the solution, proving its convergence and illustrating it by numerical experiments.
\end{abstract}


\end{frontmatter}

\section{Introduction}
\label{sec1}
Fractional differential equations extend the classical theory of differential equations by incorporating derivatives of non-integer order, providing a powerful tool for modelling memory and hereditary effects in various physical and engineering problems. The growing interest in fractional calculus comes from its ability to more accurately describe complex phenomena where the traditional integer-order models would fail. In contrast to ordinary differential equations, which imply local interactions, fractional differential equations show the ability to capture long-range dependencies and anomalous diffusion, making them highly relevant in fields such as viscoelasticity, fluid dynamics, control theory, signal processing, finance, and many others \cite{BDST16,Hi00,Ma06,Ma10,Za02}. \\
In this paper we study nonlinear fractional boundary problems that have the Riesz-Caputo derivative as the fractional operator.  Among various fractional operators, the Riesz-Caputo operator plays a crucial role due to its unique ability to model diffusion processes with directional symmetry \cite{MK00}, such as anomalous transport in porous media and turbulence in fluid mechanics. The symmetric property of the Riesz-Caputo operator was also exploited in \cite{Pit:22} to efficiently evaluate the Riesz-Caputo derivative of B-splines, and later applied in \cite{Sor15} to the numerical solution of a linear convection-diffusion problem. Fractional differential equations involving the Riesz-Caputo operator are, however, a rather new research subject, so analytical results on existence and uniqueness of their solutions as well as numerical methods for their approximations are yet limited. In \cite{Di10,KST06},  comprehensive studies of differential problems with the Riemann-Liouville and the Caputo fractional derivatives are provided.  Pedas and Tamme have studied both theoretically and numerically different
types of linear fractional differential problems based on the Caputo
operator, see \cite{Pedas2011,Pedas2011_2,Pedas2012} and references therein,
and later extended the analysis to nonlinear fractional initial value
problems \cite{Pedas2014}. In these works the differential problem is
reformulated as a weakly singular integral equation, and the convergence of
the collocation method is proved within the framework developed by Vainikko
\cite{Vainikko93,Brunner}. Nonlinear fractional differential problems based
on the classical Caputo operator have been considered also in
\cite{Pitolli_nl1,Pitolli_nl2}. Our aim is to extend these results to the Riesz-Caputo operator: the boundary value problem is reduced to a singular integral equation of Cauchy type on the interval, whose closed-form inversion is classical \cite{Samko,Gakhov,Tricomi}. This yields an explicit integral representation of the solution for an arbitrary non-integer order $\alpha>0$ and for general linear two-point boundary conditions, from which existence and uniqueness follow by the Banach fixed point theorem.  The analysis covers right-hand sides of the form $f(x,y)=D^{\alpha}_{RC}w(x)+f_1(x,y)$, with $w$ a given smooth function: this includes the right-hand sides generated by smooth solutions, which are H\"older continuous of order exactly $m-\alpha$ at the endpoints and therefore fall outside the class of H\"older continuous right-hand sides of order greater than $m-\alpha$ required by the inversion formula. Since this representation contains a singular operator of Cauchy type, the
convergence analysis does not follow the framework of \cite{Vainikko93}:
we bound the collocation residual (consistency) and combine it with the
stability of the continuous problem with respect to perturbations of the
right-hand side, which the closed-form inversion makes available.  We stress what is new and what is inherited from previous works: the integral representation of Section \ref{subsec2}, valid for any $m-1<\alpha<m$, the treatment of the general boundary conditions, and the convergence analysis of Section \ref{subsec33} are new; the computational machinery (optimal B-spline bases, collocation at the Greville abscissae, exact evaluation of the Riesz-Caputo derivative of B-splines) is taken from \cite{Pi18-b,PPP:20,Pit:22,Sor15}.\\
The paper is organized as follows: In Section \ref{sec2} we introduce the mathematical model along with some useful definitions and a few fundamental lemmas describing analytical properties of the fractional integral and differential operators involved. In the same section we also prove existence and uniqueness of the solution under suitable hypotheses and we discuss the behaviour of the solution at the endpoints. In Section \ref{sec3} we describe the numerical method and provide a formal convergence analysis.  The theoretical findings will be demonstrated by numerical experiments in Section \ref{sec4}. Finally in Section \ref{sec5} we outline the conclusions.

\section{Mathematical model}
\label{sec2}
We consider the fractional boundary value problem:
\begin{equation}
\label{eq:main_eq}
\left \{ \begin{array}{ll}
 D_{RC}^\alpha u(x)=f(x,u(x)),  \ \ \ x \in [0,L], \\ \\
a_i u^{(i)}(0)+b_iu^{(i)}(L)=c_i, \ \ i=0,..,{m-1} ,
\end{array}
\right.
\end{equation}
where $m-1<\alpha<m$, $m=\lceil \alpha \rceil \in \mathbb{N}$, and $a_i,b_i,c_i\in\RR$ with $a_i+b_i\neq 0$ for $i=0,\dots,m-1$. Here, $D_{RC}^\alpha u(x)$ denotes the Riesz-Caputo fractional derivative defined as \cite{Ag07,Di10,Po99}
\begin{equation}
\label{eq:RCdef}
D_{RC}^\alpha u(x) \coloneqq \frac{1}{2}(D^\alpha_{0,x} + (-1)^m \ D^\alpha_{x,L})u(x),
\end{equation}
where the left and right Caputo derivatives are defined by:
\begin{flalign}
\label{eq:sinistro}
D^\alpha_{0,x} u(x) &\coloneqq \frac{1}{\Gamma(\beta)} \int_0^x \frac{u^{(m)}(\xi)}{(x-\xi)^{1-\beta}}  d\xi, \\
\label{eq:destro}
D^\alpha_{x,L} u(x)  &\coloneqq \frac{(-1)^m }{\Gamma(\beta)} \int_x^L \frac{u^{(m)}(\xi)}{(\xi-x)^{1-\beta}}  d\xi,
\end{flalign}
with {$m= \lceil \alpha \rceil$}, $\beta \coloneqq m-\alpha \in (0,1)$ and $\Gamma$ the Euler gamma function.\\
Note that in definition \eqref{eq:RCdef} we use the constant $\frac{1}{2}$ in front of the combination of the left and right Caputo operator because in this case the Riesz-Caputo derivative agrees with the ordinary derivative when $\alpha$ is an integer. However other choices can be found in literature (see e.g.  \cite{Rieszcos}).

\subsection{Existence and uniqueness of solutions}
\label{subsec2}
We start by analysing some properties of the Riesz-Caputo operator.

\begin{lemma} \label{properties}
Let $m-1<\alpha<m$ and let $ D_{RC}^\alpha (\cdot)$ denote the Riesz-Caputo operator defined in \eqref{eq:RCdef}. Let $p_i(x)= k x^i$, $i=0,\dots,m-1$, $ k \in \RR$. Then,
\begin{equation*}
\label{eq:prop_constant}
D_{RC}^\alpha (p_i)=0, \qquad i=0,\dots,m-1.
\end{equation*}
\end{lemma}
\proof
To compute the Riesz-Caputo derivative, we have to evaluate the integrals in eqs. \eqref{eq:sinistro}-\eqref{eq:destro}. Since $m-1<\alpha<m$, in the integrands the $m$-th derivative appears, which is zero for all polynomials of degree at most $m-1$.
\proofend

\smallskip

Now, we introduce the fractional left and right Riemann-Liouville integrals of order $\mu$:
\begin{equation}
\label{eq:int_left}
{I}_{0,x}^{(\mu)} u(x) := \frac{1}{\Gamma(\mu)} \int_0^x \frac{u(\xi)}{(x-\xi)^{1-\mu}} d\xi,
\end{equation}
\begin{equation}
\label{eq:int_right}
{I}_{x,L}^{(\mu)} u(x) :=\frac{1}{\Gamma(\mu)} \int_x^L \frac{u(\xi)}{(\xi-x)^{1-\mu}} d\xi,
\end{equation}
for $\mu>0$, and we set ${I}_{0,x}^{0}u = {I}_{x,L}^{0}u = u$, i.e.  ${I}_{0,x}^{0} = {I}_{x,L}^{0} =I$, are the identity operators. \\
We also introduce the two-sided (Riesz-type) fractional integral:
\begin{equation}
\label{eq:int_Ries}
 \IR\psi(x) := \frac{1}{\Gamma(\beta)}\int_0^L |x-\xi|^{\beta-1}\psi(\xi)\,d\xi = \bigl( I^{(\beta)}_{0,x} + I^{(\beta)}_{x,L} \bigr)\psi(x).
\end{equation}
We use a different symbol for this operator to stress that $\IR\psi$ is a function of $x$ on the whole interval, and should not be confused with the left-sided integral \eqref{eq:int_left} evaluated at $x=L$ (or the right-sided one \eqref{eq:int_right} evaluated at $x=0$).
The above defined operators are linear and bounded from $L^\infty(0,L)$ into $C(0,L)$ \cite{Brunner,Samko}, and respect the following property:
\begin{equation}
\label{eq:identity}
D^\alpha_{RC} u(x) = \frac12 \Bigl( I^{(\beta)}_{0,x} + I^{(\beta)}_{x,L} \Bigr) u^{(m)}(x) \, = \, \frac{1}{2}\, \IR u^{(m)}(x).
\end{equation}
We also recall some useful Lemmas:
\begin{lemma} \label{lemmaprop} \cite{KST06} If $Re \{\alpha\} >0$ and $u(x) \in L_p(0,L), \ 1 \leq p \leq \infty$,  then
\begin{align*}
D^\alpha_{0,x} {I}_{0,x}^{(\alpha)} u(x)=u(x),\\
D^\alpha_{x,L} {I}_{x,L}^{(\alpha)} u(x)=u(x).
\end{align*}
\end{lemma}
\begin{lemma} \label{lemma1} \cite{KST06} Let $Re \{\alpha\} >0$. If $u(x) \in C^m[0,L]$, then the following relations hold:
\begin{align}
\label{eq:ILDL}
I_{0,x}^{(\alpha)}\,D_{0,x}^\alpha  u(x)=u(x)-\sum_{l=0}^{m-1} \frac{u^{(l)}(0)}{l!}x^l,\\
\label{eq:IRDR}
I_{x,L}^{(\alpha)}\,D_{x,L}^\alpha  u(x)=u(x)-\sum_{l=0}^{m-1} \frac{(-1)^lu^{(l)}(L)}{l!}(L-x)^l.
\end{align}
\end{lemma}

We shall also need the left-sided Riemann-Liouville derivative of order $\mu\in(0,1)$,
\begin{equation}
\label{eq:RLder}
\mathfrak{D}^{\mu}_{0,x} v(x) := \frac{1}{\Gamma(1-\mu)}\,\frac{d}{dx}\int_0^x \frac{v(\xi)}{(x-\xi)^{\mu}}\,d\xi,
\end{equation}
which differs from the Caputo derivative \eqref{eq:sinistro} by a boundary term: if $v\in C[0,L]$ has an integrable derivative, then $\mathfrak{D}^{\mu}_{0,x}v(x) = \frac{v(0)}{\Gamma(1-\mu)}x^{-\mu} + \frac{1}{\Gamma(1-\mu)}\int_0^x (x-\xi)^{-\mu}v'(\xi)\,d\xi$ \cite{KST06}.

By \eqref{eq:identity} problem \eqref{eq:main_eq} is equivalent to the two-sided integral equation of the first kind
\begin{equation}
\label{eq:twosided}
\Bigl( I^{(\beta)}_{0,x} + I^{(\beta)}_{x,L} \Bigr) \varphi(x) = 2 f(x,u(x)), \qquad \varphi := u^{(m)},
\end{equation}
supplemented with the boundary conditions in \eqref{eq:main_eq}. Equation \eqref{eq:twosided} is (up to the factor $\Gamma(\beta)$) the Carleman equation with kernel $|x-\xi|^{\beta-1}$ studied in \cite[\S 30.4]{Samko}. In order to invert it, we shall use the following composition relation between the left Riemann-Liouville derivative and the right Riemann-Liouville integral.

\begin{lemma}[{Composition relation \cite[Thm.~11.5]{Samko}}]
\label{lem:composition}
Let $\beta\in(0,1)$ and let $\psi \in L_p(0,L)$ with $p>1/\beta$. Then, for every
$x\in(0,L)$,
\begin{equation}
\label{eq:composition}
\mathfrak{D}^{\beta}_{0,x}\, I^{(\beta)}_{x,L} \psi(x)
= \cos(\beta\pi)\,\psi(x)
+ \sin(\beta\pi)\,\bigl(S_\beta\psi\bigr)(x),
\qquad
\bigl(S_\beta\psi\bigr)(x) := \frac{1}{\pi}\,\mathrm{p.v.}\!\int_0^L
\Bigl(\frac{\xi}{x}\Bigr)^{\!\beta}\frac{\psi(\xi)}{\xi-x}\,d\xi .
\end{equation}
\end{lemma}
\noindent Identity \eqref{eq:composition} is \cite[Eq.~(11.14)]{Samko} with $a=0$, $b=L$,
and $S_\beta$ is the finite Hilbert transform conjugated with the weight $x^{\beta}$.
It follows from the relation
$I^{(\beta)}_{x,L}\psi = \cos(\beta\pi)\,I^{(\beta)}_{0,x}\psi
+ \sin(\beta\pi)\,I^{(\beta)}_{0,x} S_\beta\psi$ \cite[Eq.~(11.16)]{Samko}, valid for
$\psi\in L_p(0,L)$, $p>1$, by applying $\mathfrak{D}^{\beta}_{0,x}$ and using
$\mathfrak{D}^{\beta}_{0,x}I^{(\beta)}_{0,x}\psi=\psi$ \cite[Thm.~2.4]{Samko}; the condition
$p>1/\beta$ ensures that it holds at every point of $(0,L)$.

\begin{remark}
\label{rem:RLnotCaputo}
It is essential that the operator applied in \eqref{eq:composition} is the
Riemann--Liouville derivative \eqref{eq:RLder} and not the Caputo derivative
\eqref{eq:sinistro}: the two differ by a boundary term which does not vanish in
general (see the formula following \eqref{eq:RLder}), and only the former is an exact
left inverse of $I^{(\beta)}_{0,x}$. We also stress that \eqref{eq:composition}
connects the left- and right-sided operators through a \emph{singular operator of
Cauchy type}, and not through a multiplicative constant: this is precisely what
prevents $D^{\alpha}_{RC}$ from being rewritten as a single one-sided fractional
operator acting on $u^{(m)}$, and what makes the inversion of \eqref{eq:twosided}
genuinely two-sided.
\end{remark}

\begin{lemma}[Inversion of the two-sided equation]
\label{lem:inversion}
Let $\beta\in(0,1)$ and let $g \in H^{\lambda}[0,L]$ for some $\lambda>\beta$, so that
$\mathfrak{D}^{\beta}_{0,x}g$ exists and $\rho\,\mathfrak{D}^{\beta}_{0,\cdot}g \in H^{*}$
\cite[Lemma~13.1]{Samko}. Then the equation
\[
\Bigl( I^{(\beta)}_{0,x} + I^{(\beta)}_{x,L} \Bigr)\varphi(x) = 2 g(x), \qquad x\in(0,L),
\]
has a unique solution in the class of functions with at most weak singularities of order $\beta/2$ at the endpoints, given by
\begin{equation}
\label{eq:phi_formula}
\varphi(x) \;=\; \mathfrak{D}^{\beta}_{0,x} g(x) \;-\; \tan\!\Bigl(\frac{\beta\pi}{2}\Bigr)\, \frac{1}{\rho(x)}\, \bigl( S\,[\rho\, \mathfrak{D}^{\beta}_{0,\cdot}\, g]\bigr)(x),
\end{equation}
where $S$ denotes the finite Hilbert transform on $(0,L)$,
\[
(S\psi)(x) := \frac{1}{\pi}\,\mathrm{p.v.}\!\int_0^L \frac{\psi(\xi)}{\xi-x}\, d\xi .
\]
\end{lemma}

\proof
Applying $\mathfrak{D}^{\beta}_{0,x}$ to the equation in the statement and using $\mathfrak{D}^{\beta}_{0,x}I^{(\beta)}_{0,x}\varphi=\varphi$, valid for every summable
$\varphi$ \cite[Thm.~2.4]{Samko}, together with
Lemma \ref{lem:composition}, we obtain the dominant singular integral equation of Cauchy type
\begin{equation}
\label{eq:dominant}
\bigl(1+\cos(\beta\pi)\bigr)\varphi(x) + \frac{\sin(\beta\pi)}{\pi}\,\mathrm{p.v.}\!\int_0^L \Bigl(\frac{\xi}{x}\Bigr)^{\!\beta}\frac{\varphi(\xi)}{\xi-x}\,d\xi
= 2\,\mathfrak{D}^{\beta}_{0,x} g(x).
\end{equation}
Multiplying \eqref{eq:dominant} by $x^\beta$ and setting $\mu(\xi):=\xi^\beta \varphi(\xi)$, equation \eqref{eq:dominant} takes the standard form
\begin{equation}
\label{eq:dominant_std}
A_0\, \mu(x) + \frac{B_0}{\pi}\,\mathrm{p.v.}\!\int_0^L \frac{\mu(\xi)}{\xi-x}\,d\xi = 2\, x^\beta\, \mathfrak{D}^{\beta}_{0,x}g(x),
\end{equation}
with constant coefficients
\[
A_0 = 1+\cos(\beta\pi) = 2\cos^2\!\Bigl(\frac{\beta\pi}{2}\Bigr), \qquad B_0 = \sin(\beta\pi), \qquad A_0^2+B_0^2 = 4\cos^2\!\Bigl(\frac{\beta\pi}{2}\Bigr).
\]
The closed-form solution of \eqref{eq:dominant_std} is classical \cite{Gakhov,Tricomi,Samko}. Writing $A_0+iB_0 = 2\cos(\tfrac{\beta\pi}{2})\,e^{i\beta\pi/2}$, the characteristic exponent of the equation is $\theta = \beta/2$, and the fundamental (canonical) function associated with the solution class of functions having at most weak (integrable) singularities of order $\beta/2$ at both endpoints is
\[
w(x) = \Bigl(\frac{x}{L-x}\Bigr)^{\!\beta/2}.
\]
In this class the index of the equation is zero, so that the solution exists and is unique (no free constants and no orthogonality conditions appear). This is the natural solution class here, since it corresponds to $\varphi=u^{(m)}$ having at most integrable endpoint singularities. It is the class $H^{*}$ of H\"olderian functions with integrable endpoint singularities \cite[Def.~13.1]{Samko}, in which the Carleman equation is uniquely solvable for right-hand sides in $H^*_\beta\supset H^{\lambda}[0,L]$, $\lambda>\beta$ \cite[Thm.~30.7]{Samko}. The solution of \eqref{eq:dominant_std} in this class is
\[
\mu(x) = \frac{1}{A_0^2+B_0^2}\Bigl[ A_0\, F(x) - \frac{B_0}{\pi}\, w(x)\, \mathrm{p.v.}\!\int_0^L \frac{F(\xi)}{w(\xi)\,(\xi-x)}\, d\xi \Bigr],
\]
with $F(x) := 2x^\beta \mathfrak{D}^{\beta}_{0,x}g(x)$. Transforming back to $\varphi = x^{-\beta}\mu$, all the weights combine into the remarkably symmetric function
\begin{equation}
\label{eq:rho}
\rho(\xi) := \bigl( \xi\,(L-\xi) \bigr)^{\beta/2},
\end{equation}
since $x^{-\beta}w(x)^{\phantom{1}} = \rho(x)^{-1}$ and $\xi^{\beta}w(\xi)^{-1} = \rho(\xi)$, and the trigonometric coefficients simplify to $\tfrac{2A_0}{A_0^2+B_0^2}=1$ and $\tfrac{2B_0}{\pi(A_0^2+B_0^2)} = \tfrac{1}{\pi}\tan(\tfrac{\beta\pi}{2})$. \\
Conversely, $\varphi$ given by \eqref{eq:phi_formula} solves
\eqref{eq:dominant} and, belonging to the solution class above, lies in
$L^p(0,L)$ for some $p>1/\beta$. By Lemma~\ref{lem:composition},
\eqref{eq:dominant} then states that
$(I^{(\beta)}_{0,x}+I^{(\beta)}_{x,L})\varphi-2g$ belongs to the null space
of $\mathfrak{D}^\beta_{0,x}$, i.e.\ equals $c\,x^{\beta-1}$
\cite{KST06,Samko}; since $p>1/\beta$, it is bounded by H\"older's
inequality, hence $c=0$.
\proofend

The inversion formula also gives information on the behavior of $\varphi$ at the endpoints, which
is the property that governs the accuracy attainable on a uniform mesh.

\begin{remark}[Endpoint behaviour]
\label{rem:endpoint}
The solution of Lemma \ref{lem:inversion} belongs to $H^{*}$, hence it has at most integrable
singularities at the endpoints. Their typical order is $\beta/2$ at \emph{both} endpoints, in
agreement with the canonical function $w(x)=(x/(L-x))^{\beta/2}$ of the dominant equation
\eqref{eq:dominant_std}. Indeed, writing
\begin{equation}
\label{eq:weighted}
 \varphi(x) \;=\; \bigl(x(L-x)\bigr)^{-\beta/2}\,\varrho_g(x),
\end{equation}
the explicit solutions of \cite{Samko} show that for $g(x)=x^n$, $n=0,1,2,\dots$, the
function $\varrho_g$ is a polynomial of degree $n$ which does not vanish at $x=0$ and $x=L$; in
these cases the exponent $-\beta/2$ depends only on $\beta$ and not on the datum. The singularity
may however be absent for particular data: every H\"older continuous $\varphi$ is the solution
corresponding to $g=\frac12 \IR\varphi$, and then $\varrho_g$ vanishes at the
endpoints. Such a $g$ need not belong to $H^{\lambda}[0,L]$ with $\lambda>\beta$: if $\varphi$ does
not vanish at an endpoint, $g$ is H\"older continuous of order exactly $\beta$ there, so that this
observation does not contradict the hypotheses of Lemma \ref{lem:inversion} (see also the
discussion preceding Hypothesis \ref{hyp:fw}). We do not prove that \eqref{eq:weighted} holds with a H\"older continuous $\varrho_g$
for every admissible datum; this would require an asymptotic analysis of \eqref{eq:phi_formula} at
the endpoints, which is not needed in the sequel.
\end{remark}

\begin{remark}[Symmetric form and cancellation]
\label{rem:symmetric}
Formula \eqref{eq:phi_formula} singles out the left endpoint. Let $(Rv)(x)=v(L-x)$. Since
\[
 R\rho=\rho,\qquad R\,\mathfrak D^{\beta}_{0,\cdot}R=\mathfrak D^{\beta}_{\cdot,L},\qquad SR=-RS,
\]
conjugating \eqref{eq:phi_formula} with $R$ yields the mirror identity
\begin{equation}
\label{eq:phi_formula_right}
 \varphi \;=\; \mathfrak D^{\beta}_{x,L}g \;+\; \tan\Bigl(\frac{\beta\pi}{2}\Bigr)\frac{1}{\rho}\,
 S\bigl[\rho\,\mathfrak D^{\beta}_{\cdot,L}g\bigr],
\end{equation}
and the half-sum of \eqref{eq:phi_formula} and \eqref{eq:phi_formula_right} treats the two
endpoints alike. Although the two terms in \eqref{eq:phi_formula} may separately exhibit stronger
endpoint singularities (each of them behaves like $x^{-\beta}$ as $x\to0^+$ when $g(0)\ne0$), their
leading singular parts cancel; hence the endpoint behaviour of $\varphi$ must be inferred from the
solution class of the singular integral equation, cf.\ Remark \ref{rem:endpoint}, rather than
from termwise estimates. This is also what makes the solution class compatible with Lemma
\ref{lem:composition}: a function with endpoint singularities of order $\beta/2$ belongs to
$L_p(0,L)$ for $p<2/\beta$, so that the admissible range $1/\beta<p<2/\beta$ is nonempty precisely
because the exponent is $\beta/2$ and not $\beta$. A numerical illustration of the cancellation is
given in Section \ref{subsec41}.
\end{remark}

The weight itself provides the simplest instance of this behaviour: being the equilibrium
density of the Riesz kernel on the interval \cite{Landkof}, it solves the equation with a constant
right-hand side, and does so with a constant available in closed form \cite{Samko}. This provides an exact test
solution with the genuine endpoint behaviour, which manufactured solutions cannot supply.

\begin{proposition}
\label{prop:exactpair}
For every $\beta\in(0,1)$ and every $L>0$,
\begin{equation}
\label{eq:exactpair}
 \frac12 \IR\Bigl[\bigl(\cdot\,(L-\cdot)\bigr)^{-\beta/2}\Bigr](x)
 \;=\; \frac{\Gamma(\beta/2)\,\Gamma(1-\beta/2)}{2\,\Gamma(\beta)}
 \;=\; \frac{\pi}{2\,\Gamma(\beta)\,\sin(\pi\beta/2)},
 \qquad x\in(0,L),
\end{equation}
constant in $x$ and independent of $L$.
\end{proposition}
\proof
Identity \eqref{eq:exactpair} is equivalent to the case $n=0$ of the explicit solutions of the
Carleman equation with right-hand side $(x-a)^n$ listed in \cite{Samko}, which gives
$\varphi(x)=\pi^{-1}\sin(\beta\pi/2)\,\bigl(x(L-x)\bigr)^{-\beta/2}$ for the right-hand side $1$. We
include a short proof based on the classical equilibrium property of the Riesz kernel.

\emph{Reduction to $[-1,1]$.} The change of variables $x=\frac L2(1+\tau)$, $\xi=\frac L2(1+\sigma)$
produces the factors $(L/2)^{\beta-1}$ from the kernel $|x-\xi|^{\beta-1}$, $(L/2)^{-\beta}$ from the
weight $\bigl(\xi(L-\xi)\bigr)^{-\beta/2}$ and $L/2$ from the measure, whose product is $1$. Hence
\[
 \frac12 \IR\Bigl[\bigl(\cdot\,(L-\cdot)\bigr)^{-\beta/2}\Bigr](x)
 = \frac{1}{2\Gamma(\beta)}\,U(\tau),
 \qquad U(\tau):=\int_{-1}^{1}|\tau-\sigma|^{\beta-1}(1-\sigma^2)^{-\beta/2}\,d\sigma ,
\]
which does not depend on $L$.

\emph{Constancy.} $U$ is the potential, with respect to the Riesz kernel $|\tau-\sigma|^{\beta-1}$, of
the density $(1-\sigma^2)^{-\beta/2}$. Up to normalization, this is the equilibrium density of the
interval $[-1,1]$ for that kernel, so that $U$ is constant on $(-1,1)$ \cite[Ch.~II]{Landkof}.

\emph{Value of the constant.} It suffices to evaluate $U$ at $\tau=0$. By symmetry and the
substitution $u=\sigma^2$,
\[
 \begin{aligned}
 U(0) &= 2\int_0^1\sigma^{\beta-1}(1-\sigma^2)^{-\beta/2}\,d\sigma
 = \int_0^1 u^{\beta/2-1}(1-u)^{-\beta/2}\,du\\
 &= B\Bigl(\frac\beta2,\,1-\frac\beta2\Bigr)
 = \Gamma\Bigl(\frac\beta2\Bigr)\Gamma\Bigl(1-\frac\beta2\Bigr).
 \end{aligned}
\]
Dividing by $2\Gamma(\beta)$ gives the first expression in \eqref{eq:exactpair}; the second one
follows from the reflection formula $\Gamma(z)\Gamma(1-z)=\pi/\sin(\pi z)$ with $z=\beta/2$.
\proofend

\begin{remark}
\label{rem:samko_check}
\sloppy
Equation \eqref{eq:twosided}, multiplied by $\Gamma(\beta)$, is the classical Carleman equation
\cite[Eq.~(30.83)]{Samko}, whose unique solution in $H^{*}$ is given by
\cite[Thm.~30.7]{Samko}. By uniqueness, \eqref{eq:phi_formula} is an equivalent closed form of
that solution. Its advantage is that the $m$-fold integration can be carried out explicitly, which
is what allows the representation of Lemma \ref{lemma12} to be written in closed form for arbitrary
$m$.
\end{remark}

We can now state the integral representation of the solution of problem \eqref{eq:main_eq}.

\begin{lemma}[Integral representation] \label{lemma12}
Let $m-1<\alpha<m$, $\beta=m-\alpha$, $a_i,b_i,c_i \in \RR$ with $a_i+b_i \neq 0$ for $i=0,\dots,m-1$, and assume $f:[0,L]\times \RR \rightarrow \RR$ continuous and such that $g_u:=f(\cdot,u(\cdot))\in H^{\lambda}[0,L]$ for some $\lambda>\beta$. Then $\mathfrak{D}^{\beta}_{0,x}g_u$ exists and the weighted function $\rho\,\mathfrak{D}^{\beta}_{0,\cdot}g_u$ belongs to $H^*$ \cite[Lemma~13.1]{Samko}.
 Then $u \in C^{m-1}[0,L]$ with $u^{(m)}$ having at most weak endpoint singularities of order $\beta/2$ is a solution of \eqref{eq:main_eq} if and only if it satisfies the integral equation
\begin{equation}
\label{eq:reform12}
u(x) \;=\; V(x;u) \;+\; \sum_{r=0}^{m-1} C_r\, x^r,
\end{equation}
where
\begin{equation}
\label{eq:Vdef}
V(x;u) \;:=\; I^{(\alpha)}_{0,x} f(\cdot,u(\cdot))(x)
\;-\; \tan\!\Bigl(\frac{\beta\pi}{2}\Bigr)\, I^{(m)}_{0,x}\Bigl[ \frac{1}{\rho}\, S\bigl[\rho\, \mathfrak{D}^{\beta}_{0,\cdot} f(\cdot,u(\cdot))\bigr]\Bigr](x),
\end{equation}
with $\rho$ as in \eqref{eq:rho}, $I^{(m)}_{0,x}$ the $m$-fold integral $I^{(m)}_{0,x}\psi(x) = \frac{1}{(m-1)!}\int_0^x (x-\xi)^{m-1}\psi(\xi)\,d\xi$, and where the constants $C_0,\dots,C_{m-1}$ are determined by the boundary conditions through the (triangular) recursion
\begin{equation}
\label{eq:Crec}
\begin{split}
C_r \;=\; \frac{1}{r!\,(a_r+b_r)} \Bigl[ c_r \;-\; b_r\, V^{(r)}(L;u) \;-\; b_r \sum_{j=r+1}^{m-1} C_j\, \frac{j!}{(j-r)!}\, L^{\,j-r} \Bigr],&\\
r = m-1, m-2, \dots, 0.&
\end{split}
\end{equation}
\end{lemma}

\proof
The proof consists of three steps: inversion of the two-sided equation, integration of the two terms of the inversion formula, and recovery of the boundary constants.

Let $u$ be a solution of \eqref{eq:main_eq} and set $\varphi := u^{(m)}$, $g(x):=f(x,u(x))$. By \eqref{eq:identity}, $\varphi$ solves the two-sided equation of Lemma \ref{lem:inversion}, hence $\varphi$ is given by \eqref{eq:phi_formula}. Integrating $m$ times from $0$, we obtain $u = I^{(m)}_{0,x}\varphi + P_{m-1}$ for some polynomial $P_{m-1}(x)=\sum_{r=0}^{m-1}C_r x^r$. It remains to integrate the two terms of \eqref{eq:phi_formula}.

\emph{First term.} Since $g$ is bounded, $\bigl(I^{(1-\beta)}_{0,x}g\bigr)(0)=0$, hence $I^{(\beta)}_{0,x}\mathfrak{D}^{\beta}_{0,x} g = g$ \cite[Thm.~2.4]{Samko}, and by the semigroup property of the left Riemann-Liouville integrals \cite[Thm.~2.5]{Samko},
\[
I^{(m)}_{0,x}\,\mathfrak{D}^{\beta}_{0,x} g \;=\; I^{(m-\beta)}_{0,x}\bigl( I^{(\beta)}_{0,x}\mathfrak{D}^{\beta}_{0,x} g \bigr) \;=\; I^{(\alpha)}_{0,x} g ,
\]
which is the first term in \eqref{eq:Vdef}.

\emph{Second term.} Set
\[
 h(x):=\rho(x)\,\mathfrak{D}^{\beta}_{0,x}g(x),\qquad \Theta(x):=\rho(x)^{-1}\,(Sh)(x).
\]
Since $g\in H^{\lambda}[0,L]$ with $\lambda>\beta$, the Riemann--Liouville derivative admits the
decomposition \cite[Lemma~13.1]{Samko} (for such $g$ the Riemann--Liouville and Marchaud derivatives coincide \cite[Thm.~13.1]{Samko})
\begin{equation}
\label{eq:RLdecomp}
\mathfrak{D}^{\beta}_{0,x} g(x) = \frac{g(0)}{\Gamma(1-\beta)}\,x^{-\beta} + \tilde g(x),
\qquad
\tilde g(x)=\frac{g(x)-g(0)}{\Gamma(1-\beta)\,x^{\beta}}
+\frac{\beta}{\Gamma(1-\beta)}\int_0^x\frac{g(x)-g(t)}{(x-t)^{1+\beta}}\,dt ,
\end{equation}
where $\tilde g\in H^{\lambda-\beta}[0,L]$, $\tilde g(0)=0$ and, denoting by $[g]_\lambda$ the H\"older seminorm of $g$,
\begin{equation}
\label{eq:gtilde}
 |\tilde g(x)|\le \frac{\lambda}{(\lambda-\beta)\,\Gamma(1-\beta)}\,[g]_\lambda\,x^{\lambda-\beta},
 \qquad x\in[0,L].
\end{equation}
Hence $h(x)=O(x^{-\beta/2})$ as $x\to0^+$ and $h(x)=O\bigl((L-x)^{\beta/2}\bigr)$ as $x\to L^-$; in
particular $h$ belongs to the weighted H\"older class $H^*$ and, since $\beta<1$, to $L^2(0,L)$, with
\begin{equation}
\label{eq:hL2}
 \|h\|_{L^2(0,L)}\le \frac{|g(0)|}{\Gamma(1-\beta)}\,\bigl\|x^{-\beta/2}(L-x)^{\beta/2}\bigr\|_{L^2(0,L)}
 + \|\rho\|_\infty L^{1/2}\,\|\tilde g\|_\infty \;\le\; c_h\,\|g\|_{H^\lambda[0,L]} .
\end{equation}
The finite Hilbert transform is bounded on $L^2(0,L)$ with norm at most $1$ \cite[Thm.~11.3]{Samko}, and
$\rho^{-1}\in L^2(0,L)$ because $\beta<1$. By the Cauchy--Schwarz inequality,
\begin{equation}
\label{eq:ThetaL1}
 \|\Theta\|_{L^1(0,L)} \le \|\rho^{-1}\|_{L^2(0,L)}\,\|Sh\|_{L^2(0,L)}
 \le \|\rho^{-1}\|_{L^2(0,L)}\,\|h\|_{L^2(0,L)} \le c_\Theta\,\|g\|_{H^\lambda[0,L]},
\end{equation}
with $c_\Theta$ depending only on $\beta$, $\lambda$ and $L$. Hence $\Theta\in L^1(0,L)$ and
$I^{(m)}_{0,x}\Theta$ is a well-defined $m$-fold antiderivative of $\Theta$. We note in passing that
$\Theta(x)=O(x^{-\beta})$ as $x\to0^+$ \cite[Thm.~11.1]{Samko}, in agreement with Remark
\ref{rem:symmetric}; this pointwise information is not needed in the sequel.

\emph{Boundary conditions.} Since $V = I^{(\alpha)}_{0,x} g - \tan(\tfrac{\beta\pi}{2}) I^{(m)}_{0,x}\Theta$ with $g\in L^\infty$ and $\Theta\in L^1$, for $0\le r\le m-1$ we have $V^{(r)} = I^{(\alpha-r)}_{0,x}g - \tan(\tfrac{\beta\pi}{2}) I^{(m-r)}_{0,x}\Theta$, which is continuous on $[0,L]$ and satisfies
\begin{equation}
\label{eq:Vr0}
V^{(r)}(0;u) = 0, \qquad r=0,\dots,m-1,
\end{equation}
because $I^{(\mu)}_{0,x}\psi(0)=0$ for every $\mu>0$ and $\psi \in L^1$ (for the second term, $m-r\geq 1$ and the integral of an $L^1$ function is absolutely continuous and vanishes at $0$). Imposing $a_r u^{(r)}(0)+b_r u^{(r)}(L)=c_r$ and using \eqref{eq:Vr0} together with $\bigl(\sum_j C_j x^j\bigr)^{(r)}(0) = r!\,C_r$ and $\bigl(\sum_j C_j x^j\bigr)^{(r)}(L) = \sum_{j\geq r} C_j \frac{j!}{(j-r)!}L^{j-r}$, we obtain, for $r=m-1,\dots,0$,
\[
a_r\, r!\, C_r + b_r\Bigl( V^{(r)}(L;u) + \sum_{j= r}^{m-1} C_j\, \frac{j!}{(j-r)!}\,L^{j-r} \Bigr) = c_r,
\]
which is an upper triangular linear system in $C_{m-1},\dots,C_0$ with diagonal entries $r!\,(a_r+b_r)\neq 0$; solving downwards yields \eqref{eq:Crec}. The converse implication follows by reading the above computations backwards: if $u$ satisfies \eqref{eq:reform12}, then $u\in C^{m-1}$, $u^{(m)}=\varphi$ as in \eqref{eq:phi_formula}, so that $D^\alpha_{RC}u = \frac12 \IR\varphi = f(\cdot,u)$ by Lemma \ref{lem:inversion}, and the boundary conditions hold by construction.
\proofend

For $m=2$, \eqref{eq:reform12}--\eqref{eq:Crec} reduce to the explicit representation of the
classical two-point second-order case; the corresponding constants are computed in Appendix
\ref{appA}.

The representation \eqref{eq:reform12} expresses the solution as the fixed point of the (nonlinear) operator
\begin{equation}
\label{eq:Toperator}
(\mathcal{T}u)(x) := V(x;u) + \sum_{r=0}^{m-1} C_r(u)\, x^r ,
\end{equation}
where the dependence of $V$ and of the $C_r$'s on $u$ is through $g_u:=f(\cdot,u(\cdot))$ only, and this dependence is affine in $g_u$. The following mapping property will be used both for the uniqueness and for the convergence analysis. From now on $\lambda\in(\beta,1]$ is fixed.

From now on we assume $m\ge 2$, i.e.\ $\alpha>1$; the case $0<\alpha<1$ is discussed in Remark
\ref{rem:m1}. On $C^1[0,L]$ we use the norm
\[
 \Cnorm{v} := \|v\|_\infty + \|v'\|_\infty ,
\]
and we denote by $B_R:=\{v\in C^1[0,L]:\ \Cnorm{v}\le R\}$ the closed ball of radius $R$, which is
a complete metric space. We recall that for $0<\lambda\le 1$ and $v\in C^1[0,L]$
\begin{equation}
\label{eq:holder_C1}
 \|v\|_{H^\lambda[0,L]} \le \|v\|_\infty + L^{1-\lambda}\|v'\|_\infty \le c_L\,\Cnorm{v},
 \qquad c_L:=\max(1,L^{1-\lambda}).
\end{equation}

\begin{lemma}[Mapping property]
\label{lem:mapping}
Let $m\ge2$, $\beta=m-\alpha\in(0,1)$, and let $g\in H^{\lambda}[0,L]$. Define
\[
 \Psi g := V[g] + \sum_{r=0}^{m-1} C_r[g]\,x^r ,
\]
where $V[g]$ is given by \eqref{eq:Vdef} with $f(\cdot,u(\cdot))$ replaced by $g$, and the
$C_r[g]$ are defined by \eqref{eq:Crec}. Then
$\Psi: H^{\lambda}[0,L]\to C^{m-1}[0,L]\subset C^1[0,L]$ is affine and continuous. More precisely,
there exists $K>0$, depending only on $\alpha$, $\lambda$, $L$ and $\{a_i,b_i\}$, such that
\begin{equation}
\label{eq:mapping}
 \Cnorm{\Psi g_1 - \Psi g_2} \;\le\; K\,\|g_1-g_2\|_{H^\lambda[0,L]} .
\end{equation}
\end{lemma}
\proof
Since $\Psi$ is affine, $\Psi g_1-\Psi g_2=\Psi_0 g$ with $g:=g_1-g_2$, where $\Psi_0$ denotes the
linear map obtained by setting $c_r=0$ in \eqref{eq:Crec}. Thus
\[
 \Psi_0 g = I^{(\alpha)}_{0,x}g - \tan\Bigl(\frac{\beta\pi}{2}\Bigr)\,I^{(m)}_{0,x}\Theta_g
 + \sum_{r=0}^{m-1} C_r[g]\,x^r,
 \qquad \Theta_g:=\rho^{-1}S\bigl[\rho\,\mathfrak D^\beta_{0,\cdot}g\bigr].
\]

\emph{The function $\Theta_g$.} By \eqref{eq:ThetaL1}, the map $g\mapsto\Theta_g$ is linear and
continuous from $H^{\lambda}[0,L]$ into $L^1(0,L)$:
$\|\Theta_g\|_{L^1(0,L)}\le c_\Theta\|g\|_{H^\lambda[0,L]}$.

\emph{Derivatives and coefficients.} For every $r=0,\dots,m-1$ we have
$\alpha-r\ge1-\beta>0$ and $m-r\ge1$, so that
$V^{(r)}[g] = I^{(\alpha-r)}_{0,x}g - \tan(\tfrac{\beta\pi}{2})\,I^{(m-r)}_{0,x}\Theta_g$ and
\[
 \|V^{(r)}[g]\|_\infty \le \frac{L^{\alpha-r}}{\Gamma(\alpha-r+1)}\|g\|_\infty
 +\tan\Bigl(\frac{\beta\pi}{2}\Bigr)\frac{L^{m-r-1}}{(m-r-1)!}\|\Theta_g\|_{L^1}
 \le C_V\bigl(\|g\|_\infty+\|\Theta_g\|_{L^1}\bigr).
\]
Since the recursion \eqref{eq:Crec} is triangular and $a_r+b_r\ne0$, a backward induction on
$r=m-1,\dots,0$, with this bound at $x=L$, gives
$\max_{r}|C_r[g]| \le C_C\bigl(\|g\|_\infty+\|\Theta_g\|_{L^1}\bigr)$.

\emph{Conclusion.} Collecting the bounds for $r=0,1$ and those for the coefficients,
$\Cnorm{\Psi_0 g}\le C\bigl(\|g\|_\infty+\|\Theta_g\|_{L^1}\bigr)$. Using
$\|g\|_\infty\le\|g\|_{H^\lambda}$ and \eqref{eq:ThetaL1}, we obtain
$\Cnorm{\Psi_0 g}\le K\|g\|_{H^\lambda[0,L]}$, which is \eqref{eq:mapping}.
\proofend

\begin{hypothesis}
\label{hyp:f}
There exist $R>0$ and $L_f>0$ such that
\begin{itemize}
\item[(i)] $x\mapsto f(x,w(x))$ belongs to $H^{\lambda}[0,L]$ for every $w\in B_R$;
\item[(ii)] for all $u,v\in B_R$,
\begin{equation}
\label{eq:LipC1}
 \|f(\cdot,u(\cdot))-f(\cdot,v(\cdot))\|_{H^\lambda[0,L]} \;\le\; L_f\,\Cnorm{u-v}.
\end{equation}
\end{itemize}
\end{hypothesis}

\begin{remark}
\label{rem:hypf}
The $C^1$ norm on the right-hand side of \eqref{eq:LipC1} cannot in general be replaced by the
uniform norm. Indeed, even for
$f(x,y)=y$, an estimate of the form $\|u-v\|_{H^\lambda[0,L]}\le C\|u-v\|_\infty$ cannot hold
uniformly on $B_R$: for $u-v=\varepsilon\sin(x/\varepsilon)$ the right-hand side is
$C\varepsilon$, while the H\"older seminorm on the left is of order $\varepsilon^{1-\lambda}$.

A sufficient condition for Hypothesis \ref{hyp:f} is that $f$ be H\"older continuous of order
$\lambda$ in $x$, uniformly for $|y|\le R$, and that $f_y$ exist and be H\"older continuous of
order $\lambda$ on $[0,L]\times[-R,R]$. Indeed, for $w\in B_R$,
\[
 |f(x,w(x))-f(y,w(y))| \le |f(x,w(x))-f(y,w(x))| + |f(y,w(x))-f(y,w(y))| ,
\]
where the first term is bounded by a multiple of $|x-y|^\lambda$ and the second one by
$\|f_y\|_\infty R\,|x-y|$, so that $f(\cdot,w(\cdot))\in H^\lambda[0,L]$. Moreover, for $u,v\in B_R$,
\[
 f(x,u(x))-f(x,v(x)) = a(x)\,(u-v)(x),
 \qquad a(x)=\int_0^1 f_y\bigl(x,v(x)+s(u-v)(x)\bigr)\,ds .
\]
The assumed regularity of $f_y$ implies $\|a\|_{H^\lambda}\le M_R$ uniformly for $u,v\in B_R$.
Hence, using the product estimate in $H^\lambda$ and \eqref{eq:holder_C1},
\[
 \|f(\cdot,u)-f(\cdot,v)\|_{H^\lambda} \le 2M_R\,c_L\,\Cnorm{u-v},
\]
so that \eqref{eq:LipC1} holds with $L_f=2M_Rc_L$.
\end{remark}

\begin{theorem}
\label{teounique}
Let $m\ge2$, let the assumptions of Lemma \ref{lemma12} hold, and let Hypothesis \ref{hyp:f} be
satisfied. Assume
\begin{equation}
\label{eq:smallness}
 q:=K L_f < 1
 \qquad\text{and}\qquad
 \Cnorm{\Psi[f(\cdot,0)]} \le (1-q)\,R ,
\end{equation}
where $K$ is the constant of Lemma \ref{lem:mapping}. Then problem \eqref{eq:main_eq} has a unique
solution $u\in B_R$. Moreover, $u\in C^{m-1}[0,L]$.
\end{theorem}
\proof
By Lemma \ref{lemma12}, $u$ solves \eqref{eq:main_eq} if and only if
$u=\mathcal T u:=\Psi[f(\cdot,u(\cdot))]$. For $u,v\in B_R$, Lemma \ref{lem:mapping} and Hypothesis
\ref{hyp:f} give
\[
 \Cnorm{\mathcal T u-\mathcal T v} \le K\,\|f(\cdot,u)-f(\cdot,v)\|_{H^\lambda} \le q\,\Cnorm{u-v},
\]
and, since $q<1$, $\mathcal T$ is a contraction on $B_R$. Moreover, for $u\in B_R$,
\[
 \Cnorm{\mathcal T u}\le \Cnorm{\mathcal T u-\mathcal T 0}+\Cnorm{\mathcal T 0}
 \le q\,\Cnorm{u}+(1-q)R\le R ,
\]
by \eqref{eq:smallness}, so that $\mathcal T(B_R)\subseteq B_R$. Since $B_R$ is complete, the Banach
fixed point theorem yields a unique fixed point $u\in B_R$, and $u\in C^{m-1}[0,L]$ by Lemma
\ref{lem:mapping}.
\proofend

\paragraph{Extension to data generated by smooth solutions.}
Hypothesis \ref{hyp:f}(i) requires the right-hand side $x\mapsto f(x,u(x))$ to be H\"older continuous of order
$\lambda>\beta$. This excludes a relevant class of problems. If $u$ is smooth up to the boundary,
then by \eqref{eq:identity}
\[
 D^\alpha_{RC}u(x)=D^\alpha_{RC}u(0)+\frac{u^{(m)}(0)}{2\,\Gamma(\beta+1)}\,x^{\beta}+O(x)
 \qquad\text{as }x\to0^+,
\]
and similarly at $x=L$. Hence, unless $u^{(m)}$ vanishes at both endpoints, the right-hand side is H\"older
continuous of order exactly $\beta$. The idea is to separate this borderline $\beta$-H\"older
contribution, generated by a smooth reference function $w$, and to apply the previous theory to the
translated unknown $v=u-w$.

\begin{hypothesis}
\label{hyp:fw}
The right-hand side has the form
\begin{equation}
\label{eq:fsplit}
 f(x,y)=D^\alpha_{RC}w(x)+f_1(x,y),
\end{equation}
where $w\in C^{m}[0,L]$ is given and $f_1:[0,L]\times\RR\to\RR$ is continuous. Moreover, there exist
$R_0>0$ and $L_f>0$ such that, with $w+B_{R_0}:=\{w+v:\ v\in B_{R_0}\}$,
\begin{itemize}
\item[(i)] $x\mapsto f_1(x,z(x))$ belongs to $H^{\lambda}[0,L]$ for every $z\in w+B_{R_0}$;
\item[(ii)] for all $z_1,z_2\in w+B_{R_0}$,
\begin{equation}
\label{eq:LipC1w}
 \|f_1(\cdot,z_1(\cdot))-f_1(\cdot,z_2(\cdot))\|_{H^\lambda[0,L]} \;\le\; L_f\,\Cnorm{z_1-z_2}.
\end{equation}
\end{itemize}
\end{hypothesis}

For $w=0$, Hypothesis \ref{hyp:fw} reduces to Hypothesis \ref{hyp:f} with $R=R_0$. We denote by
$\widetilde\Psi$ the solution map of Lemma \ref{lem:mapping} with the boundary data $c_r$ in
\eqref{eq:Crec} replaced by
\begin{equation}
\label{eq:ctilde}
 \tilde c_r:=c_r-a_r\,w^{(r)}(0)-b_r\,w^{(r)}(L),\qquad r=0,\dots,m-1 .
\end{equation}
Since the estimate \eqref{eq:mapping} involves only the linear part $\Psi_0$, it holds for
$\widetilde\Psi$ with the same constant $K$.

\begin{corollary}
\label{cor:shift}
Let $m\ge2$, $a_i+b_i\ne0$ for $i=0,\dots,m-1$, and let Hypothesis \ref{hyp:fw} hold. Assume
\begin{equation}
\label{eq:smallnessw}
 q:=KL_f<1\qquad\text{and}\qquad \Cnorm{\widetilde\Psi[f_1(\cdot,w(\cdot))]}\le(1-q)\,R_0 ,
\end{equation}
where $K$ is the constant of Lemma \ref{lem:mapping}. Then problem \eqref{eq:main_eq} has a unique
solution $u\in w+B_{R_0}$. Moreover $u=w+v$, where $v\in B_{R_0}\cap C^{m-1}[0,L]$ is the unique
fixed point in $B_{R_0}$ of $v\mapsto\widetilde\Psi[f_1(\cdot,w+v)]$.
\end{corollary}
\proof
Since $w\in C^m[0,L]$, $D^\alpha_{RC}w$ is well defined and continuous, and $D^\alpha_{RC}$ is
linear. Hence $u=w+v$ solves \eqref{eq:main_eq} if and only if $v$ solves
\[
 D^\alpha_{RC}v(x)=\tilde f(x,v(x)),\qquad \tilde f(x,y):=f_1(x,w(x)+y),
\]
with the boundary conditions $a_rv^{(r)}(0)+b_rv^{(r)}(L)=\tilde c_r$, $r=0,\dots,m-1$; moreover
$u^{(m)}$ and $v^{(m)}$ differ by the continuous function $w^{(m)}$, so that the two problems are
posed in the same class of solutions. By Hypothesis \ref{hyp:fw}, $\tilde f$ satisfies Hypothesis
\ref{hyp:f} with $R=R_0$ and the same $L_f$, and \eqref{eq:smallnessw} is \eqref{eq:smallness} for
the problem in $v$, because $\tilde f(\cdot,0)=f_1(\cdot,w(\cdot))$. The statement follows from
Theorem \ref{teounique} applied to $v$.
\proofend

\begin{remark}
\label{rem:shift}
The decomposition \eqref{eq:fsplit} is an assumption on the right-hand side, not a modification of the
problem: it enlarges the admissible right-hand sides by the term $D^\alpha_{RC}w$, which is H\"older
continuous of order $\beta$ (Lemma \ref{lem:holder} below) but in general not of any higher order. In
particular it covers the problems with a prescribed smooth solution $u^\star$ and right-hand side
$f(x,y)=D^\alpha_{RC}u^\star(x)+g(u^\star(x))-g(y)$ used in Section \ref{subsec43}: there
$w=u^\star$, $f_1(x,y)=g(u^\star(x))-g(y)$, so that $f_1(\cdot,w)=0$, $\tilde c_r=0$, and the second
condition in \eqref{eq:smallnessw} is trivially satisfied. We also note that the constant $K$ is not
available in explicit form, so that the condition $q<1$ is a smallness assumption on the
nonlinearity which cannot be checked on a given example.
\end{remark}

\section{Numerical method}
\label{sec3}

\subsection{Optimal B-spline basis functions}
\label{subsec31}
A spline is a piecewise polynomial function widely used in computational mathematics for tasks such as curve fitting, data smoothing, and numerical approximation of integral and differential equations.  On the real line splines can be written as linear combination of the translate of cardinal B-splines, which are compactly supported piecewise polynomials having breakpoints at the integers. The cardinal B-spline $N_n$ of integer degree $n>0$ can be defined through the truncated power function $x^n_+ := \max(0,x)^n$ as
\begin{equation*}
	N_{n}(x) = (n+1) [0,1,\ldots,n+1] (y-x)^n_+,
\end{equation*}
where  $[x_0,x_1,\ldots,x_{n+1}] f$ denotes the~divided difference of the function $f$ on the~sequence of knots $\{x_0,x_1,\ldots,x_{n+1}\}$.  This definition works also for the case $n=0$, but we will consider only the case of (at least) continuous polynomials.\\
On a~finite interval $[0,L]$, $L \in \mathbb{N}$, splines can be represented as linear combination of the~optimal basis,  which have integer knots of multiplicity one inside the interval and knots of multiplicity $n+1$ at the~endpoints. Optimality refers to a set of properties, such as compact support, high smoothness, endpoint conditions, partition of unity and numerical stability, that make these bases ideal for  approximating functions. For more details on splines see the basic monographs \cite{dB78,Sc07}.

The~optimal basis ${\cal N}_n = \left\{N_{j,n}(x), 0\le j \le L+n-1 \right\}$ on the~interval $[0,L]$ has $L+n$ basis functions. The $2n$ functions $N_{j,n}$ and~$N_{L+n-1-j,n}$, $0 \le j \le n-1$, are respectively the left and~right boundary functions. The $L-n$ functions $N_{j,n}$, $n \le j \le L-1$, are the internal functions (see  \cite{Pi18-b} for more details).

We~observe that the~optimal basis is centrally symmetric and fulfills the endpoint conditions
\begin{equation*}
	\label{eq:endpoint}
	N_{j,n}(0) = \delta_{j0}\,, \quad N_{L+n-1-j,n}(L) = \delta_{j0}\,, \qquad 0 \le j \le L+n-1,
\end{equation*}
where $\delta_{j0}$ denotes the~Kronecker symbol.
The~optimal basis ${\cal N}_n$ can be refined by dilation \cite{dB78,Sc07}. Let $h$ be the~refinement step, that is, the~distance between the~refined knot sequence:
\begin{equation*}
\begin{cases}
	x_{0,h} = x_{1,h}=\cdots =x_{n-1,h} = 0; \\
	x_{j,h} = (j-n+1)h\,, \text{ for } n\le j \le M_h-2; \\
	x_{M_h-1,h}=x_{M_h,h}=\cdots=x_{M_h+n-2,h}=L\,,
	\end{cases}
\end{equation*}
where $M_h = L/h+n$ is the number of basis functions. In this sequence the endpoints have multiplicity $n$: it is the one entering the definition of the Greville abscissae below, and it is obtained from the knot sequence of the basis, in which the endpoints have multiplicity $n+1$, by dropping the first and the last knot. The~refined basis ${\cal N}_{h,n} = \{N_{j,h,n}(x), \ 0 \le j \le M_h-1\}$, $x \in [0,L]$, has $L/h-n$ internal functions and $2n$ boundary functions, so that $L/h-n+2n=M_h$. We notice that refining the knots only leads to an increase in the number of internal functions, while the number of boundary functions remains unchanged.

We will use these basis functions to represent the approximate solution of the fractional differential problem \eqref{eq:main_eq}. Throughout this section we assume $n \ge m$, so that the $m$-th derivative of the spline is well defined as a bounded piecewise polynomial on $[0,L]$, and the boundary conditions of order up to $m-1$ can be imposed. Another great advantage of using such a representation is that fractional derivatives of the optimal B-spline basis functions can be computed analytically using the explicit formula provided in \cite{Pi18-b,Pit:22} and by leveraging the symmetry properties of the Riesz-Caputo operator and the basis functions together. Specifically,
 \begin{proposition}[\cite{Pit:22}]
	\label{th_symm}
	Let $\Phi = \{\phi_j(x),0\le j \le M\}$ be a centrally symmetric basis function in the interval $[0,L]$, i.e.,
		$\phi_j(x) = \phi_{M-j}(L-x),$
	for $0 \le j \le M$ and  $x \in [0,L]$.
	Then, the left and right Caputo derivatives of the basis function $\Phi$ satisfy the following symmetry properties:
	\begin{equation*} \label{eq:symm_derphi}
	D^\alpha_{x,L}\phi_{j}(x)=D^\alpha_{0,L-x} \phi_{j}(L-x), \qquad 0 \le j \le M.
	\end{equation*}
 \end{proposition}

\subsection{Collocation method and Greville Abscissae}
\label{subsec32}
In \cite{Sor15} we discussed the advantages of using Greville Abscissae as collocation points, and compared this choice with other possible distributions of points.  Motivated by the effectiveness of choosing Greville Abscissae, also shown by numerical experiments, in this paper we will adopt the same option.  Greville Abscissae are defined as \cite{dB78,Sc07,Sb05}:
\begin{equation}
\label{eq:Grevpoint}
\eta_{p,h,n}=\frac{x_{p,h}+...+x_{p+n-1,h}}{n}, \qquad 0\le p \le M_h-1.
\end{equation}
In particular $\eta_{0,h,n}=0$ and $\eta_{M_h-1,h,n}=L$. The Greville abscissae are not uniformly spaced: their spacing is $h$ in the interior and smaller near the endpoints, where the knots are multiple; for $n=3$, for instance, the successive spacings are $\frac h3,\frac{2h}3,h,\dots,h,\frac{2h}3,\frac h3$.
We solve the fractional boundary value problem \eqref{eq:main_eq} by the spline collocation method introduced in \cite{PPP:20}. First, we approximate the solution to the differential problem by a spline function:
\begin{equation*} \label{eq:spline_approx}
	u(x) \approx u_{h,n}(x) =  \sum_{k=0}^{M_h-1} \ell_{k,h,n} N_{k,h,n}(x), \qquad x \in [0,L]\,.
\end{equation*}
Then, by collocating the differential problem \eqref{eq:main_eq} at the internal Greville Abscissae \eqref{eq:Grevpoint} with indices $\lceil m/2 \rceil \le p \le M_h -1- \lfloor m/2 \rfloor$, and imposing the $m$ boundary conditions, we obtain the square nonlinear system
\begin{equation}
\label{eq:coll_eq}
\left \{ \begin{array}{ll}
\displaystyle D_{RC}^\alpha u_{h,n}(\eta_{p,h,n})=f\bigl(\eta_{p,h,n},\, u_{h,n}(\eta_{p,h,n})\bigr), & \lceil m/2 \rceil \le p \le M_h -1- \lfloor m/2 \rfloor, \\[2mm]
\displaystyle a_i\, u_{h,n}^{(i)}(0)+b_i\, u_{h,n}^{(i)}(L)=c_i, & i=0,\dots,m-1,
\end{array}
\right.
	\end{equation}
where the coefficients $\{\ell_{k,h,n}\}$ are the $M_h$ unknowns. Note that the number of internal collocation equations is $M_h-1-\lfloor m/2\rfloor - \lceil m/2\rceil +1 = M_h-m$, which together with the $m$ boundary equations gives a square system of dimension $M_h$; for $m=2$ this reduces to the natural choice $0<p<M_h-1$. The boundary equations enforce the boundary conditions of \eqref{eq:main_eq} exactly. Moreover, since the Greville Abscissae have spacing at most $h$ and only a number of points independent of $h$ is discarded near the endpoints, there exists a constant $c_{n,m}>0$, depending only on $n$ and $m$, such that every $x\in[0,L]$ is within distance $c_{n,m}\,h$ of one of the collocation points used in \eqref{eq:coll_eq}.

\begin{remark}
\label{rem:degree}
The requirement $n\ge m$ is essential. If $n=m$, then $u_{h,n}^{(m)}$ is bounded and piecewise
constant, so that the collocation scheme and the basic convergence result (Theorem \ref{teo:conv})
remain valid, although the refined estimate of Proposition \ref{prop:refined} does not apply. If
$n\ge m+1$, then $u_{h,n}^{(m)}$ is continuous (and Lipschitz), and both convergence results apply.
For $n<m$, the boundary conditions of order $m-1$ and the Caputo derivatives
\eqref{eq:sinistro}--\eqref{eq:destro} are not defined in the classical setting considered here.
Hence splines of degree at least $m$ are required, while $n\ge m+1$ is preferable.
\end{remark}

\subsection{Convergence analysis}
\label{subsec33}

In this section we prove the convergence of the collocation method \eqref{eq:coll_eq}, following the classical stability + consistency paradigm. The structure is the following: we first show that the collocation residual is small (consistency, Lemma \ref{lem:consist}), and then that the exact problem is stable with respect to perturbations of the right-hand side (Lemmas \ref{lem:mapping} and \ref{lem:linstab}), so that a small residual implies a small error.

Throughout the section we denote by $u\in C^{m-1}[0,L]$ the exact solution of \eqref{eq:main_eq} provided by Theorem \ref{teounique} or, more generally, by Corollary \ref{cor:shift}, and by $u_{h,n}$ a solution of the collocation system \eqref{eq:coll_eq}. We define the residual
\begin{equation}
\label{eq:residual}
r_{h}(x) \;:=\; D^{\alpha}_{RC}\, u_{h,n}(x) \;-\; f\bigl(x, u_{h,n}(x)\bigr), \qquad x\in[0,L].
\end{equation}
By construction $r_h$ vanishes at the interior Greville abscissae $\eta_{p,h,n}$, $\lceil m/2\rceil\le p\le M_h-1-\lfloor m/2\rfloor$, at which the equation is collocated, and the boundary conditions are satisfied by $u_{h,n}$ exactly.

We start with an elementary H\"older estimate for the Riesz--Caputo derivative of functions with bounded $m$-th derivative, which quantifies the smoothness of the first term in \eqref{eq:residual}.

\begin{lemma}[H\"older continuity of $D^\alpha_{RC}$]
\label{lem:holder}
Let $v\in C^{m-1}[0,L]$ with $v^{(m)}\in L^\infty(0,L)$. Then for all $x_1,x_2 \in [0,L]$,
\begin{equation}
\label{eq:holder}
\bigl| D^\alpha_{RC} v(x_1) - D^\alpha_{RC} v(x_2) \bigr| \;\le\; C_\beta\, \| v^{(m)} \|_{L^\infty}\, |x_1-x_2|^{\beta}, \qquad C_\beta = \frac{2}{\Gamma(\beta+1)},
\end{equation}
with $\beta = m-\alpha \in (0,1)$. In particular $D^\alpha_{RC}v\in H^{\beta}[0,L]$.
\end{lemma}
\proof
By \eqref{eq:identity}, $D^\alpha_{RC}v = \frac12\bigl( I^{(\beta)}_{0,x} + I^{(\beta)}_{x,L}\bigr) v^{(m)}$. Let $0\le x_1 < x_2 \le L$, $\psi := v^{(m)}$ and $M:=\|\psi\|_{L^\infty}$. For the left-sided integral,
\[
\bigl| I^{(\beta)}_{0,x_2}\psi - I^{(\beta)}_{0,x_1}\psi \bigr|
\le \frac{M}{\Gamma(\beta)} \Bigl[ \int_0^{x_1} \bigl( (x_1-t)^{\beta-1} - (x_2-t)^{\beta-1} \bigr)\, dt + \int_{x_1}^{x_2} (x_2-t)^{\beta-1}\, dt \Bigr]
\]
\[
= \frac{M}{\Gamma(\beta+1)} \bigl[ x_1^\beta - x_2^\beta + 2 (x_2-x_1)^\beta \bigr] \le \frac{2M}{\Gamma(\beta+1)} (x_2-x_1)^\beta,
\]
since $x_1^\beta-x_2^\beta\le0$. The right-sided integral satisfies the same bound by symmetry, and averaging the two estimates yields \eqref{eq:holder}.
\proofend

As is customary for collocation methods for nonlinear problems \cite{Pedas2014,Brunner,Vainikko93}, the existence of solutions of the discrete system \eqref{eq:coll_eq} and their uniform boundedness are assumed a priori (see Remark \ref{rem:H}).

\begin{hypothesis}
\label{hyp:H}
There exist $h_0>0$ and $R>0$ such that for every $h \le h_0$ the collocation system \eqref{eq:coll_eq} admits a solution $u_{h,n}$ with
\begin{equation}
\label{eq:uhball}
 \Cnorm{u_{h,n}-w}\le R_0 \qquad\text{and}\qquad \| u_{h,n}^{(m)} \|_{L^\infty} \le R ,
\end{equation}
where $w$ and $R_0$ are those of Hypothesis \ref{hyp:fw} (in the setting of Theorem \ref{teounique}, $w=0$ and $R_0$ is the radius in Hypothesis \ref{hyp:f}), and $u_{h,n}^{(m)}$ is understood almost everywhere.
\end{hypothesis}

The next lemma provides the stability of the linear part of the problem: it bounds a function satisfying homogeneous boundary conditions in terms of its Riesz--Caputo derivative. Its proof relies on a second form of the inversion, due to \cite[Eq.~(30.85)]{Samko}, which does not involve derivatives of the right-hand side. Formally, inversion of the two-sided fractional integral costs $\beta$ derivatives, while recovering $v$ from $v^{(m)}$ gains $m$ integrations, resulting in a net smoothing of order $m-\beta=\alpha$.

\begin{lemma}[Linear stability]
\label{lem:linstab}
Let $m\ge2$, $\beta=m-\alpha$, and let $\{a_i,b_i\}$ satisfy $a_i+b_i\ne0$, $i=0,\dots,m-1$. Let
$v\in C^{m-1}[0,L]$ satisfy the homogeneous boundary conditions
$a_i v^{(i)}(0)+b_i v^{(i)}(L)=0$, $i=0,\dots,m-1$, and assume that
\begin{itemize}
\item[(a)] $v^{(m)}=\psi_1+\psi_2$, where $\psi_1\in L^\infty(0,L)$ and $\psi_2\in H^*$ has at most weak
endpoint singularities of order $\beta/2$;
\item[(b)] $r:=D^\alpha_{RC}v$ is H\"older continuous on $[0,L]$, of some order $\mu>0$.
\end{itemize}
Then, for every $p>1/(1-\beta)$,
\begin{equation}
\label{eq:linstab}
 \Cnorm{v} \;\le\; K_p\,\bigl\|D^\alpha_{RC}v\bigr\|_{L^p(0,L)} ,
\end{equation}
with a constant $K_p>0$ depending only on $\alpha$, $p$, $L$ and $\{a_i,b_i\}$.
\end{lemma}
Assumptions (a) and (b) are tailored to the error function occurring in the convergence proof: its
$m$-th derivative is the sum of a bounded spline contribution and of a term belonging to the
solution class of Lemma \ref{lem:inversion}.
\proof
Set $\varphi:=v^{(m)}\in L^1(0,L)$ and $\Phi:=I^{(\beta)}_{0,x}\varphi$. By \eqref{eq:identity},
$\bigl(I^{(\beta)}_{0,x}+I^{(\beta)}_{x,L}\bigr)\varphi=2r$.

\emph{Step 1: a singular equation for $\Phi$.} By \cite[Eq.~(11.18)]{Samko}, valid for
$\varphi\in L_p(0,L)$, $p\ge1$,
\[
 I^{(\beta)}_{x,L}\varphi=\cos(\beta\pi)\,\Phi+\sin(\beta\pi)\,(L-x)^{\beta}\,
 S\bigl[(L-\cdot)^{-\beta}\Phi\bigr](x).
\]
Hence $\widetilde\Phi:=(L-x)^{-\beta}\Phi$ solves the dominant singular equation
\begin{equation}
\label{eq:domPhi}
 \bigl(1+\cos(\beta\pi)\bigr)\widetilde\Phi(x)+\sin(\beta\pi)\,(S\widetilde\Phi)(x)
 =\frac{2\,r(x)}{(L-x)^{\beta}},\qquad x\in(0,L),
\end{equation}
which is \cite[Eq.~(30.47)]{Samko} with constant coefficients $u=v=1$.

\emph{Step 2: its solution.} By (a), $\Phi=I^{(\beta)}_{0,x}\psi_1+I^{(\beta)}_{0,x}\psi_2$. The
first term is H\"older continuous of order $\beta$ on $[0,L]$ and vanishes at $x=0$, by the estimate
of the left-sided integral in the proof of Lemma \ref{lem:holder}. The second one is H\"older
continuous in $[0,L)$, vanishes at $x=0$ like $x^{\beta/2}$, and has at most an arbitrarily weak
singularity at $x=L$, by the mapping properties of $I^{(\beta)}_{0,x}$ in weighted H\"older spaces
\cite[Thm.~13.13]{Samko}. Hence $\widetilde\Phi=(L-x)^{-\beta}\Phi$ belongs to the class of
functions in $H^*$ which are bounded at $x=0$. By (b), the right-hand side of \eqref{eq:domPhi} has the form
required in \cite[Thm.~30.2]{Samko}. In this class the index of \eqref{eq:domPhi} is zero,
and its unique solution is \cite[Eq.~(30.85)]{Samko}
\begin{equation}
\label{eq:Gdef}
 I^{(\beta)}_{0,x}\varphi(x)=G(x),\qquad
 G:=r-\tan\Bigl(\frac{\beta\pi}{2}\Bigr)\,\rho\,S\Bigl[\frac{r}{\rho}\Bigr],
\end{equation}
with $\rho$ the weight \eqref{eq:rho}.

\emph{Step 3: representation of $v$.} By the semigroup property \cite[Thm.~2.5]{Samko},
$I^{(m)}_{0,x}\varphi=I^{(m-\beta)}_{0,x}I^{(\beta)}_{0,x}\varphi=I^{(\alpha)}_{0,x}G$, so that
\begin{equation}
\label{eq:vrepr}
 v=I^{(\alpha)}_{0,x}G+\sum_{r=0}^{m-1}C_r\,x^r ,
\end{equation}
where the $C_r$ are given by the recursion \eqref{eq:Crec} with $c_r=0$ and $V^{(r)}(L;u)$ replaced
by $I^{(\alpha-r)}_{0,x}G(L)$.

\emph{Step 4: estimates.} The operator $g\mapsto\rho\,S[g/\rho]$ is bounded in $L^p(0,L)$ if and only
if $S$ is bounded in $L_p$ with the weight $\rho^{p}$, whose exponents at the endpoints are
$p\beta/2$. By \cite[Thm.~11.3]{Samko} this holds for $-1<p\beta/2<p-1$, i.e.\ for
$p>2/(2-\beta)$, and then $\|G\|_{L^p}\le c_p\|r\|_{L^p}$. Moreover, for $k=0,\dots,m-1$ we have
$\alpha-k\ge1-\beta>0$, and H\"older's inequality gives, for $x\in[0,L]$,
\[
 \bigl|I^{(\alpha-k)}_{0,x}G(x)\bigr|\le\frac{1}{\Gamma(\alpha-k)}
 \Bigl(\int_0^x (x-t)^{(\alpha-k-1)p'}\,dt\Bigr)^{1/p'}\|G\|_{L^p}\le c\,\|G\|_{L^p},
\]
provided $(\alpha-k-1)p'>-1$. The most restrictive case is $k=m-1$, which requires $\beta p'<1$,
i.e.\ $p>1/(1-\beta)$; since $1/(1-\beta)\ge2/(2-\beta)$, all the conditions hold for
$p>1/(1-\beta)$. Therefore $|I^{(\alpha-r)}_{0,x}G(L)|\le c\|r\|_{L^p}$ for $r=0,\dots,m-1$, and a
backward induction in \eqref{eq:Crec} gives $|C_r|\le c\|r\|_{L^p}$. Differentiating
\eqref{eq:vrepr},
\[
 \|v\|_\infty\le\|I^{(\alpha)}_{0,x}G\|_\infty+\sum_r|C_r|L^r,\qquad
 \|v'\|_\infty\le\|I^{(\alpha-1)}_{0,x}G\|_\infty+\sum_{r\ge1}r|C_r|L^{r-1},
\]
and \eqref{eq:linstab} follows.
\proofend

\begin{lemma}[Consistency]
\label{lem:consist}
Let Hypotheses \ref{hyp:fw} and \ref{hyp:H} hold and let $n\ge m$. Then $r_h\in H^{\beta}[0,L]$,
\begin{equation}
\label{eq:consist}
 |r_h(x_1)-r_h(x_2)|\le C_r\,|x_1-x_2|^{\beta},\qquad
 C_r:=C_\beta\bigl(R+\|w^{(m)}\|_\infty\bigr)+L^{\lambda-\beta}\bigl(\|f_1(\cdot,w)\|_{H^\lambda}+L_fR_0\bigr),
\end{equation}
for all $x_1,x_2\in[0,L]$ and $h\le h_0$, and
\begin{equation}
\label{eq:consist2}
 \|r_h\|_\infty\le C_r\,\bigl(c_{n,m}\,h\bigr)^{m-\alpha},
\end{equation}
where $C_\beta$ is the constant of Lemma \ref{lem:holder} and $c_{n,m}$ is the covering constant of
the collocation points introduced in Section \ref{subsec32}.
\end{lemma}
\proof
By \eqref{eq:fsplit}, $r_h=D^\alpha_{RC}(u_{h,n}-w)-f_1(\cdot,u_{h,n})$. Since $n\ge m$,
$u_{h,n}^{(m)}-w^{(m)}$ is bounded by $R+\|w^{(m)}\|_\infty$, and Lemma \ref{lem:holder} gives the
first term of $C_r$. By \eqref{eq:uhball} and Hypothesis \ref{hyp:fw},
$\|f_1(\cdot,u_{h,n})\|_{H^\lambda}\le\|f_1(\cdot,w)\|_{H^\lambda}+L_fR_0$, and
$|x_1-x_2|^{\lambda}\le L^{\lambda-\beta}|x_1-x_2|^{\beta}$ because $\lambda>\beta$; this gives the
second term. Finally, every $x\in[0,L]$ is within distance $c_{n,m}h$ of a collocation point
$\eta^*$, where $r_h(\eta^*)=0$, so that \eqref{eq:consist2} follows from \eqref{eq:consist} with
$\beta=m-\alpha$.
\proofend

We can now state the main convergence result.

\begin{theorem}[Convergence]
\label{teo:conv}
Let $m\ge2$, let the hypotheses of Corollary \ref{cor:shift} (with $w=0$, those of Theorem
\ref{teounique}) and Hypothesis \ref{hyp:H} hold, let $n\ge m$ and fix any $p>1/(1-\beta)$. Then
\begin{equation}
\label{eq:convergence}
 \Cnorm{u-u_{h,n}} \;\le\; \frac{K_p\,L^{1/p}}{1-KL_f}\,\|r_h\|_\infty
 \;\le\; \frac{K_p\,L^{1/p}\,C_r\,c_{n,m}^{\,m-\alpha}}{1-KL_f}\;h^{\,m-\alpha},
\end{equation}
with $K_p$ and $C_r$ the constants of Lemmas \ref{lem:linstab} and \ref{lem:consist}.
\end{theorem}
\proof
By \eqref{eq:uhball}, $u_{h,n}\in w+B_{R_0}$, so that $f_1(\cdot,u_{h,n}(\cdot))\in H^{\lambda}[0,L]$
by Hypothesis \ref{hyp:fw}(i), and $\widetilde w_h:=w+\widetilde\Psi[f_1(\cdot,u_{h,n}(\cdot))]$ is
well defined. By Lemma \ref{lemma12} and the linearity of $D^\alpha_{RC}$,
$D^\alpha_{RC}\widetilde w_h=f(\cdot,u_{h,n})$ on $(0,L)$, and by \eqref{eq:ctilde} $\widetilde w_h$
satisfies the boundary conditions of \eqref{eq:main_eq}. Split
$u-u_{h,n}=(u-\widetilde w_h)+(\widetilde w_h-u_{h,n})$.

\emph{First term.} By Corollary \ref{cor:shift}, $u=w+\widetilde\Psi[f_1(\cdot,u)]$. Since
$u,u_{h,n}\in w+B_{R_0}$, Lemma \ref{lem:mapping} and \eqref{eq:LipC1w} give
\begin{equation}
\label{eq:term1}
 \Cnorm{u-\widetilde w_h} \le K\,\|f_1(\cdot,u)-f_1(\cdot,u_{h,n})\|_{H^\lambda} \le KL_f\,\Cnorm{u-u_{h,n}}.
\end{equation}

\emph{Second term.} Let $e_h:=u_{h,n}-\widetilde w_h$. Then $e_h\in C^{m-1}[0,L]$, it satisfies
homogeneous boundary conditions, and $D^\alpha_{RC}e_h=r_h$. Assumption (a) of Lemma
\ref{lem:linstab} holds, because $e_h^{(m)}=\bigl(u_{h,n}^{(m)}-w^{(m)}\bigr)-\psi_2$, where the first
term is bounded and $\psi_2$, the $m$-th derivative of $\widetilde\Psi[f_1(\cdot,u_{h,n})]$, is given
by Lemma \ref{lem:inversion}; assumption (b) holds by Lemma \ref{lem:consist}. Hence
\begin{equation}
\label{eq:term2}
 \Cnorm{\widetilde w_h-u_{h,n}} \le K_p\|r_h\|_{L^p(0,L)} \le K_pL^{1/p}\|r_h\|_\infty .
\end{equation}

Combining \eqref{eq:term1} and \eqref{eq:term2} and using $KL_f<1$ yields the first inequality in
\eqref{eq:convergence}; the second one is \eqref{eq:consist2}.
\proofend

The order $m-\alpha$ in \eqref{eq:convergence} does not depend on the spline degree $n$: the bottleneck in Lemma \ref{lem:consist} is the $\beta$-H\"older modulus of $D^\alpha_{RC} u_{h,n}$, not the approximation power of the spline space. For $n \ge m+1$ the consistency estimate improves everywhere except in an $O(h)$-neighbourhood of the endpoints, and the provable order saturates at (almost) one.

\begin{proposition}[Refined consistency and order]
\label{prop:refined}
In addition to the hypotheses of Theorem \ref{teo:conv}, let $n \ge m+1$, $w\in C^{m+1}[0,L]$, let $f_1$ be Lipschitz continuous in both variables on the set $\{(x,y):\ x\in[0,L],\ |y-w(x)|\le R_0\}$, and assume
\begin{itemize}
\item[$(H')$] $\ \| u_{h,n}^{(m+1)} \|_{L^\infty} \le R_1$ uniformly for $h \le h_0$.
\end{itemize}
Then, denoting by $d_x := \min(x, L-x)$ the distance of $x$ from the endpoints,
\begin{equation}
\label{eq:refined_pointwise}
| r_h(x) | \;\le\; C\, \Bigl[ \, h\, \max( d_x,\, h )^{\beta-1} \;+\; h \, \Bigr], \qquad x \in [0,L],
\end{equation}
with $C$ independent of $h$ and $x$. Consequently, for every admissible exponent $p$ in Lemma \ref{lem:linstab} with $(1-\beta)p > 1$,
\begin{equation}
\label{eq:refined_order}
\| r_h \|_{L^p(0,L)} \;\le\; C\, h^{\,\beta + 1/p},
\qquad\text{and hence}\qquad
\| u - u_{h,n} \|_{\infty} \;=\; O\bigl( h^{\, \min( 1,\; m-\alpha+1/p )} \bigr).
\end{equation}
\end{proposition}
\proof
As in Lemma \ref{lem:consist}, $r_h=D^\alpha_{RC}(u_{h,n}-w)-f_1(\cdot,u_{h,n})$. Let $\psi := u_{h,n}^{(m)}-w^{(m)}$; since $n\ge m+1$ and $w\in C^{m+1}[0,L]$, $\psi$ is Lipschitz continuous, and we denote again by $R$ and $R_1$ bounds for $\|\psi\|_\infty$ and $\|\psi'\|_{L^\infty}$ (Hypotheses \ref{hyp:H} and $(H')$). Let $\eta^*$ be the collocation point closest to $x$ and $\delta := |x-\eta^*| \le c_{n,m}h$; assume $\eta^* < x$ (the other case is symmetric). Substituting $t = x-s$, resp.\ $t=\eta^*-s$, in the left-sided integral,
\[
\bigl| I^{(\beta)}_{0,x}\psi(x) - I^{(\beta)}_{0,\eta^*}\psi(\eta^*) \bigr|
\le \frac{1}{\Gamma(\beta)} \Bigl[ \int_0^{\eta^*} s^{\beta-1} \bigl| \psi(x-s)-\psi(\eta^*-s) \bigr|\, ds
+ \int_{\eta^*}^{x} s^{\beta-1} |\psi(x-s)|\, ds \Bigr]
\]
\[
\le \frac{1}{\Gamma(\beta+1)} \Bigl[ R_1\, \delta\, L^{\beta} + R\, \bigl( x^{\beta} - (\eta^*)^{\beta} \bigr) \Bigr]
\le \frac{1}{\Gamma(\beta+1)} \Bigl[ R_1\, \delta\, L^{\beta} + R\, \min\bigl( \delta^{\beta},\, \beta\, \delta\, (\eta^*)^{\beta-1} \bigr) \Bigr],
\]
and $\min(\delta^\beta, \beta\delta(\eta^*)^{\beta-1}) \le c\, h \max(x, h)^{\beta-1}$ since $\delta \le c_{n,m}h$ and $\eta^* \ge x - c_{n,m}h$. The right-sided integral obeys the symmetric bound with $\max(L-x,h)^{\beta-1}$, and the term $f_1(\cdot,u_{h,n})$ contributes $O(h)$, because $f_1$ is Lipschitz continuous and $\|u_{h,n}'\|_\infty$ is bounded; this proves \eqref{eq:refined_pointwise}. For the $L^p$ bound, integrating \eqref{eq:refined_pointwise},
\[
\| r_h \|_{L^p}^p \le C^p \Bigl[ 2 \int_0^{h} h^{p} \, h^{(\beta-1)p}\, dx + 2\int_h^{L/2} h^p\, d^{(\beta-1)p}\, dd + L\, h^p \Bigr]
\le C'^{\,p} \Bigl[ h^{1+\beta p} + h^{p} \Bigr] \le C''^{\,p} h^{1+\beta p},
\]
where the middle integral is $O(h^{p}\, h^{\,1-(1-\beta)p}) = O(h^{1+\beta p})$ because $(1-\beta)p>1$. The error bound then follows from \eqref{eq:convergence}.
\proofend

\begin{remark}
\label{rem:saturation}
The order provable with the residual-based argument is therefore at most (almost) one, whatever the
spline degree: on a uniform grid the accuracy is limited by the weakly singular kernel near the
endpoints. This saturation is well known for one-sided fractional operators, where full order is
recovered on graded meshes \cite{Pedas2011,Pedas2012,Pedas2014}; the typical endpoint exponent
$-\beta/2$ of Remark \ref{rem:endpoint} suggests the natural grading for the Riesz-Caputo case.
\end{remark}

\begin{remark}
\label{rem:notsharp}
The estimates \eqref{eq:convergence} and \eqref{eq:refined_order} are rigorous
but, as the experiments of Section~\ref{sec4} show, far from sharp for smooth
solutions. A plausible explanation is that Theorem~\ref{teo:conv} bounds
$\|r_h\|$ in a fixed norm and then applies a stability estimate for the
\emph{continuous} problem, through the auxiliary function $\widetilde w_h$,
which the scheme never computes. In doing so it ignores the structure of the
residual: since $D^\alpha_{RC}u=f(\cdot,u)$,
\[
 r_h=\tfrac12\, \IR\bigl(u_{h,n}-u\bigr)^{(m)}
 -\bigl(f(\cdot,u_{h,n})-f(\cdot,u)\bigr),
\]
i.e.\ up to the nonlinear term the residual is a weakly singular potential of
the $m$-th derivative of the error, which oscillates on the scale of the mesh
and is therefore damped by the potential. Exploiting this structure requires
an analysis in which the discrete inverse acts directly on the discrete
residual; we leave it to future work.
\end{remark}

\begin{remark}
\label{rem:H}
Hypothesis \ref{hyp:H} plays the role of the uniform boundedness (stability) of the discrete
solutions and is standard in this context \cite{Pedas2014,Vainikko93}; it is checked a posteriori in
Section \ref{sec4}. Since the nonlinear system \eqref{eq:coll_eq} may admit several solutions, we
select the relevant one by continuation in the strength of the nonlinearity: $f$ is replaced by
$sf$, with $s$ increasing from $0$ to $1$, and each solution is used as the initial guess for the
next. An a priori proof of the solvability of \eqref{eq:coll_eq} and of Hypothesis \ref{hyp:H} is
left to future work.
\end{remark}

\begin{remark}
\label{rem:Hscope}
Hypothesis \ref{hyp:H} is not compatible with the typical endpoint behaviour described in Remark
\ref{rem:endpoint}. Indeed, if $\|u_{h,n}^{(m)}\|_{L^\infty}\le R$ for all $h\le h_0$ and
$u_{h,n}\to u$ uniformly, then the functions $u_{h,n}^{(m-1)}$ are uniformly bounded (by
interpolation between $u_{h,n}$ and $u_{h,n}^{(m)}$) and uniformly Lipschitz with constant $R$; by
the Arzel\`a--Ascoli theorem a subsequence converges uniformly, necessarily to $u^{(m-1)}$, which
is therefore Lipschitz with constant $R$. Hence $u^{(m)}\in L^\infty(0,L)$, i.e.\ the function
$\varrho_g$ of \eqref{eq:weighted} vanishes at both endpoints at least like
$\bigl(x(L-x)\bigr)^{\beta/2}$. Theorem \ref{teo:conv} and Proposition \ref{prop:refined} therefore
apply to this non-generic class of solutions, which includes all the manufactured solutions of
Section \ref{subsec43}, but not to the exact singular solution of Proposition \ref{prop:exactpair},
for which $\|u_{h,n}^{(m)}\|_\infty$ grows like $h^{-\beta/2}$ (Section \ref{subsec43}, Table \ref{tab:singularmu}). Extending the analysis
to the generic case requires a weighted version of Hypothesis \ref{hyp:H}, e.g.\
$\|\rho\,u_{h,n}^{(m)}\|_\infty\le R$, and a corresponding weighted form of Lemmas \ref{lem:holder}
and \ref{lem:consist}; this is left to future work.
\end{remark}

\begin{remark}[The case $0<\alpha<1$]
\label{rem:m1}
For $m=1$ the solution map $\Psi$ takes values only in $C[0,L]$, and the $C^1$ framework above is
not available. A fixed point argument in $H^{\lambda}[0,L]$ is possible for $\beta<\lambda<\alpha$,
hence for $\alpha>1/2$; we do not pursue this case here.
\end{remark}

\section{Numerical experiments}
\label{sec4}

The experiments are organized in three groups. The first validates the analytical results of
Section \ref{subsec2}, namely the closed-form inversion \eqref{eq:phi_formula} and the exact
solution of Proposition \ref{prop:exactpair}. The second probes the regime in which the endpoint
behaviour described in Remark \ref{rem:endpoint} is actually present, which manufactured
solutions cannot reproduce. The third measures the convergence of the collocation method
\eqref{eq:coll_eq} and compares it with Theorem \ref{teo:conv} and Proposition \ref{prop:refined}.

\subsection{Validation of the closed-form inversion}
\label{subsec41}

Two independent checks are available. The first is a test of \eqref{eq:phi_formula}:
taking $\varphi\equiv1$ and the corresponding datum $g(x)=\bigl(x^\beta+(L-x)^\beta\bigr)/
(2\Gamma(\beta+1))$, formula \eqref{eq:phi_formula} reproduces the constant to twelve significant
digits at interior points, and to nine digits at $x=2\cdot10^{-3}$, where the two terms have
individually grown to $\pm5$. This is a demanding test precisely because it exercises the
cancellation discussed in Remark \ref{rem:symmetric}.

The second check concerns Proposition \ref{prop:exactpair}. The left-hand side of
\eqref{eq:exactpair} was evaluated by a regularized quadrature, obtained from the substitution
$\xi=\frac L2(1-\cos\theta)$, under which the weight and the measure combine into
$(\sin\theta/2)^{1-\beta}$ and the only remaining singularity is the integrable one of the kernel,
itself removed by the change of variable $w=Ds^{1/\beta}$. The computed values are constant in $x$
to eleven significant digits and agree with the closed form
$\Gamma(\beta/2)\Gamma(1-\beta/2)/(2\Gamma(\beta))$ to within the quadrature accuracy, for
$\beta=0.3,\,0.5,\,0.7,\,0.9$ and four values of $x$.

\subsection{The singular regime}
\label{subsec42}

Proposition \ref{prop:exactpair} provides an exact solution which is genuinely singular at the
endpoints, and therefore a test in which the behaviour described in
Remark \ref{rem:endpoint} is present rather than suppressed. We solve the linear equation
\eqref{eq:twosided} with the constant datum of \eqref{eq:exactpair} by collocation at the Greville
abscissae, with $\varphi$ sought in the spline space $\mathcal N_{h,n}$ on a uniform mesh, and we
measure the error against the exact solution $\bigl(x(L-x)\bigr)^{-\beta/2}$. The integrals are
evaluated with the substitution described in Section \ref{subsec41}, which resolves the endpoints
exactly without local refinement.

\begin{table}[H]
\centering
\begin{tabular}{lccc}
\toprule
 & $\beta=0.3$ & $\beta=0.5$ & $\beta=0.7$ \\
\midrule
measured order in $L^2$ & $0.350$ & $0.250$ & $0.157$ \\
$(1-\beta)/2$           & $0.35$  & $0.25$  & $0.15$ \\
\addlinespace
measured order in $L^1$ & $0.838$ & $0.743$ & $0.646$ \\
$1-\beta/2$             & $0.85$  & $0.75$  & $0.65$ \\
\bottomrule
\end{tabular}
\caption{Orders of convergence for the exact singular solution of Proposition
\ref{prop:exactpair}, on a uniform mesh, $j=4,\ldots,8$ (i.e.\ $h=2^{-4},\ldots,2^{-8}$). The
measured values are stable in $j$ to three digits and coincide for $n=3$ and $n=4$.}
\label{tab:singular}
\end{table}

Two features of Table \ref{tab:singular} deserve comment. First, the orders are \emph{identical
for $n=3$ and $n=4$}: in this regime, on a uniform mesh, the spline degree plays no role. Second,
they coincide with the values dictated by the singularity alone: for $\varphi\sim d^{-\beta/2}$,
with $d$ the distance from the nearest endpoint, the first cell contributes
$\|d^{-\beta/2}\|_{L^2(0,h)}\sim h^{(1-\beta)/2}$ and
$\|d^{-\beta/2}\|_{L^1(0,h)}\sim h^{1-\beta/2}$. Away from the endpoints the behaviour is
different: on the fixed subinterval $[0.1,0.9]$ the error in the uniform norm decays with an order
that stabilizes around $1$. Together with Remark \ref{rem:saturation}, this identifies the grading
that would be required to restore the full order, and quantifies what is lost without it.

\subsection{Convergence of the collocation method}
\label{subsec43}

\paragraph{Test problems.}
We consider problem \eqref{eq:main_eq} on $[0,1]$ with $a_i=b_i=1$, $i=0,\dots,m-1$, and three
manufactured solutions $u^\star$, namely
\[
 u^\star(x)=e^{x},\qquad u^\star(x)=\sin\Bigl(\frac{\pi x}{2}\Bigr),\qquad u^\star(x)=x^{5},
\]
with right-hand side $f(x,y)=D^\alpha_{RC}u^\star(x)+g\bigl(u^\star(x)\bigr)-g(y)$, where
$g(y)=\sqrt y$, $g(y)=\tanh y$ and $g(y)=y^2$, respectively, and data
$c_i=u^{\star(i)}(0)+u^{\star(i)}(1)$. The Riesz--Caputo derivative of $u^\star$ is computed in
closed form for $x^5$ and by adaptive quadrature, after the substitution that removes the weak
singularity of the kernel, for the other two solutions. These problems have the
form \eqref{eq:fsplit} with $w=u^\star$ and $f_1(x,y)=g(u^\star(x))-g(y)$ (Remark \ref{rem:shift}),
and all three solutions have bounded $m$-th derivative, so that they belong to the class covered by
Corollary \ref{cor:shift} and Theorem \ref{teo:conv} (Remark \ref{rem:Hscope}); a solution outside
this class is considered in paragraph (d) below. Three points should be kept in mind. The smallness
condition $KL_f<1$ cannot be checked, since $K$ is not explicit. For $u^\star=e^x$ the nonlinearity
$g(y)=\sqrt y$ is smooth only for $y>0$, and the hypotheses are understood in a neighbourhood of
the range of $u^\star$, i.e.\ with $R_0<1$. Finally, since $u^{\star(m)}$ does not vanish at the
endpoints, $D^\alpha_{RC}u^\star$ is H\"older continuous of order exactly $\beta=m-\alpha$ in $x$:
this is the borderline case handled by Hypothesis \ref{hyp:fw}, and both Theorem \ref{teo:conv} and
Proposition \ref{prop:refined} apply. We take $m=2,3,4$, $\alpha=m-1+\{0.25,0.50,0.75,0.95\}$, i.e.\
$\beta=m-\alpha\in\{0.75,0.50,0.25,0.05\}$, spline degrees $n=m,\dots,m+4$, and $h=2^{-j}$,
$j=3,\dots,6$. Errors are measured in the uniform norm and in the norm $\Cnorm{\cdot}$ of Theorem
\ref{teo:conv} on a uniform grid with about four points per basis function, and the orders are
estimated from consecutive refinements.

\paragraph{(a) Errors and orders.}
Figure \ref{fig:conv} (left) shows a typical convergence history, and Table \ref{tab:orders} summarizes
the estimated orders over all test problems and all values of $m$, grouped by the difference
$n-m$ between the spline degree and the integer order, and by $\beta$.

\begin{figure}[htbp]
\centering
\begin{tikzpicture}
\begin{groupplot}[
  group style={group size=2 by 1, horizontal sep=1.9cm},
  width=0.47\textwidth, height=0.42\textwidth,
  xmode=log, ymode=log, log basis x=2,
  xtick={8,16,32,64}, xticklabels={8,16,32,64},
  xlabel={$1/h$}, xmin=6.5, xmax=150,
  tick label style={font=\footnotesize}, label style={font=\small}, title style={font=\small},
  axis line style={black!55}, tick style={black!55},
  ymajorgrids, grid style={black!10},
  legend style={font=\footnotesize, draw=none, fill=none, at={(0.03,0.03)}, anchor=south west, legend cell align=left},
  every axis plot/.append style={line width=0.9pt, mark size=2.2pt},
]
\nextgroupplot[title={$u^\star=\sin(\pi x/2)$, $\alpha=2.5$}, ylabel={$\|u^\star-u_{h,n}\|_\infty$}, ymin=5e-11, ymax=2e-3]
\addplot[serA, mark=*] coordinates {(8,3.32e-4) (16,6.53e-5) (32,1.34e-5) (64,2.79e-6)};
\addlegendentry{$n=3$}
\addplot[serB, mark=square*] coordinates {(8,3.03e-5) (16,5.66e-6) (32,1.04e-6) (64,1.89e-7)};
\addlegendentry{$n=4$}
\addplot[serC, mark=triangle*, mark size=2.8pt] coordinates {(8,2.17e-6) (16,1.26e-7) (32,6.91e-9) (64,3.79e-10)};
\addlegendentry{$n=5$}
\addplot[black!60, dashed, line width=0.5pt, mark=none, forget plot, domain=20:64] {2.79e-6*2.2*(64/x)^2.25};
\addplot[black!60, dashed, line width=0.5pt, mark=none, forget plot, domain=20:64] {1.89e-7*0.45*(64/x)^2.5};
\addplot[black!60, dashed, line width=0.5pt, mark=none, forget plot, domain=20:64] {3.79e-10*0.4*(64/x)^4};
\node[font=\scriptsize, anchor=west] at (axis cs:68,6.1e-6) {$h^{2+\beta/2}$};
\node[font=\scriptsize, anchor=west] at (axis cs:68,8.5e-8) {$h^{2+\beta}$};
\node[font=\scriptsize, anchor=west] at (axis cs:68,1.5e-10) {$h^{4}$};
\nextgroupplot[title={exact singular solution, $m=n=3$}, ymin=5e-5, ymax=1e-2]
\addplot[serA, mark=*] coordinates {(8,6.09e-3) (16,3.14e-3) (32,1.59e-3) (64,8.03e-4)};
\addlegendentry{$\beta=0.75$}
\addplot[serB, mark=square*] coordinates {(8,2.38e-3) (16,1.26e-3) (32,6.50e-4) (64,3.29e-4)};
\addlegendentry{$\beta=0.50$}
\addplot[serC, mark=triangle*, mark size=2.8pt] coordinates {(8,7.88e-4) (16,4.03e-4) (32,2.07e-4) (64,1.05e-4)};
\addlegendentry{$\beta=0.25$}
\addplot[black!60, dashed, line width=0.5pt, mark=none, forget plot, domain=20:64] {1.05e-4*0.62*(64/x)};
\node[font=\scriptsize, anchor=west] at (axis cs:68,6.5e-5) {$h$};
\end{groupplot}
\end{tikzpicture}
\caption{Errors in the uniform norm against $1/h$. Left: a smooth solution, for three spline
degrees; the dashed lines have the slopes of the empirical laws of Table \ref{tab:orders}. Right:
the exact singular solution of paragraph (d), for three values of $\beta$; the dashed line has
slope one. Theorem \ref{teo:conv} predicts order $m-\alpha=\beta$ in the first case and does not apply
in the second.}
\label{fig:conv}
\end{figure}

\begin{table}[htbp]
\centering
\small\setlength{\tabcolsep}{4pt}
\begin{tabular}{c cccc c}
\toprule
$n-m$ & $\beta=0.75$ & $\beta=0.50$ & $\beta=0.25$ & $\beta=0.05$ & empirical law \\
\midrule
0 & 2.37 {\small(2.31--2.39)} & 2.25 {\small(2.22--2.27)} & 2.13 {\small(2.10--2.16)} & 2.03 {\small(2.02--2.05)} & $2+\beta/2$ \\
1 & 2.75 {\small(2.57--3.10)} & 2.49 {\small(2.42--2.60)} & 2.23 {\small(2.20--2.26)} & 2.03 {\small(2.00--2.05)} & $2+\beta$ \\
2 & 4.48 {\small(4.24--4.85)} & 4.43 {\small(4.19--4.66)} & 4.04 {\small(3.38--4.21)} & 3.95 {\small(3.80--4.00)} & $\approx 4$ \\
\midrule
Thm.\ \ref{teo:conv} & 0.75 & 0.50 & 0.25 & 0.05 & $\beta=m-\alpha$ \\
\bottomrule
\end{tabular}
\caption{Estimated orders in the uniform norm at the finest reliable refinement: mean and range
(in brackets) over the three test problems and $m=2,3,4$. The polynomial solution $x^5$ is
excluded for $n\ge5$, where it belongs to the spline space.}
\label{tab:orders}
\end{table}

Three features stand out. First, the observed orders exceed the bound $m-\alpha$ of Theorem
\ref{teo:conv} by about two units for $n=m$ and $n=m+1$, and by about four for $n=m+2$; they also
exceed the value $1$ at which the refined estimate of Proposition \ref{prop:refined} saturates. This is the gap anticipated in
Remark \ref{rem:notsharp}. Second, for $n-m\in\{0,1\}$ the orders are remarkably regular: they
depend on $\beta$ and on the parity of $n-m$, but not on $m$ nor on the solution, and are well
described by $2+\beta/2$ for $n=m$ and by $2+\beta$ for $n=m+1$. The step of about two orders
between $n=m+1$ and $n=m+2$ is typical of collocation at the Greville abscissae. We record these
laws as empirical observations; a proof would require the sharp analysis mentioned in Remark
\ref{rem:notsharp}. Third, the orders in the norm $\Cnorm{\cdot}$ coincide with those in the
uniform norm up to a few hundredths, the only exception being $n=m=2$, where the $C^1$ order is
$2.0$ for every $\beta$. The stronger norm of Theorem \ref{teo:conv} is therefore obtained at no
cost in practice.

For $n\ge m+3$, and for $n=m+2$ on the finest mesh when $m=4$, the error reaches
$10^{-10}$--$10^{-13}$ already for $h=1/16$ or $h=1/32$ and then stagnates or grows, while the
residual of the nonlinear system increases from about $10^{-14}$ to at most $10^{-7}$. This is the
effect of rounding errors amplified by the conditioning of the collocation matrix, whose condition
number grows by a factor between about $5$ and $13$ at each halving of $h$ in our runs, and not a
loss of convergence; the corresponding orders are not
meaningful and are not reported. For $u^\star=x^5$ and $n\ge5$ the error is at the level of
rounding errors on all meshes, as expected, since $x^5$ belongs to the spline space; this
provides an independent check of the implementation.

\paragraph{(b) Verification of Hypothesis \ref{hyp:H}.}
Table \ref{tab:hypH} reports the quantities of Hypothesis \ref{hyp:H} for the least favourable
degree $n=m$, for which $u_{h,n}^{(m)}$ is piecewise constant. In all cases $\|u_{h,n}\|_\infty$
and $\|u_{h,n}'\|_\infty$ coincide with the corresponding norms of $u^\star$ to three digits on
every mesh, and $\|u_{h,n}^{(m)}\|_\infty$ converges to $\|u^{\star(m)}\|_\infty$. For $n\ge m+1$
the relative difference is below $1.5\%$ already on the coarsest mesh. Hypothesis \ref{hyp:H}
is therefore satisfied, with $R$ essentially equal to the corresponding norm of the exact
solution.

\begin{table}[H]
\centering
\begin{tabular}{l c ccc}
\toprule
$u^\star$ & $m$ & $\|u^{\star(m)}\|_\infty$ & $\|u_{h,m}^{(m)}\|_\infty$, $h=1/8$ & $\|u_{h,m}^{(m)}\|_\infty$, $h=1/64$ \\
\midrule
$e^x$ & 2 & 2.718 & 2.563 & 2.699 \\
      & 3 & 2.718 & 2.603 & 2.703 \\
      & 4 & 2.718 & 2.521 & 2.693 \\
\addlinespace
$\sin(\pi x/2)$ & 2 & 2.467 & 2.456 & 2.467 \\
                & 3 & 3.876 & 3.959 & 3.887 \\
                & 4 & 6.088 & 6.041 & 6.087 \\
\addlinespace
$x^5$ & 2 & 20  & 16.65 & 19.56 \\
      & 3 & 60  & 54.94 & 59.34 \\
      & 4 & 120 & 111.0 & 118.9 \\
\bottomrule
\end{tabular}
\caption{A posteriori check of Hypothesis \ref{hyp:H}: $m$-th derivative of the collocation
solution for $n=m$ and $\alpha=m-\frac12$, compared with that of the exact solution. The norms of
$u_{h,n}$ and $u_{h,n}'$ agree with those of $u^\star$ to three digits in all cases and are not
shown.}
\label{tab:hypH}
\end{table}

\paragraph{(c) Solution of the nonlinear system.}
The system \eqref{eq:coll_eq} is solved by a trust-region dogleg method with tolerances $10^{-12}$,
combined with the continuation in the strength of the nonlinearity described in Remark
\ref{rem:H}, in five equispaced steps $s=0,\frac14,\dots,1$, starting from the zero vector. For the
test problems above the exact solution solves the problem for every value of $s$, so that the
continuation follows the branch of $u^\star$ by construction. After each run we flag convergence
to a different root when the residual is below $10^{-8}$ while the error exceeds $10\%$ of
$\|u^\star\|_\infty$. This never occurred in the $960$ runs reported here, including those of paragraph (d). The final residual is
below $10^{-12}$ in most runs and never exceeds $1.3\cdot10^{-7}$, the largest values occurring in
the rounding-dominated regime described above, and at most $70$ iterations are needed over the
whole continuation. In about $15\%$ of the runs the solver stops because the trust-region
radius becomes too small; in all of them the residual and the error are in line with those of the
neighbouring runs, so that this termination reflects the attainable accuracy and not a failure.

\paragraph{(d) The exact singular solution.}
To illustrate Remark \ref{rem:Hscope}, we consider the problem whose solution has the typical
endpoint behaviour, namely
\[
 u^\star=I^{(m)}_{0,x}\Bigl[\bigl(t(L-t)\bigr)^{-\beta/2}\Bigr],\quad
 f(x,y)=\frac{\pi}{2\,\Gamma(\beta)\sin(\pi\beta/2)}+g\bigl(u^\star(x)\bigr)-g(y),\quad g(y)=y^2,
\]
with $D^\alpha_{RC}u^\star$ constant by Proposition \ref{prop:exactpair}; $u^\star$ and its
derivatives are evaluated through the incomplete beta function. Here the right-hand side is smooth,
so that the hypotheses of Theorem \ref{teounique} on $f$ are satisfied, but $u^{\star(m)}$ is unbounded:
Hypothesis \ref{hyp:H} cannot hold and Theorem \ref{teo:conv} does not apply. The results for
$m=2,3,4$ and $n=m,\dots,m+4$ are summarized in Table \ref{tab:singularmu}. At each halving of $h$
the norm $\|u_{h,n}^{(m)}\|_\infty$ is multiplied by $2^{\beta/2}$, i.e.\ it grows like
$h^{-\beta/2}$ as predicted, while $\|u_{h,n}\|_\infty$ and $\|u_{h,n}'\|_\infty$ converge to those
of $u^\star$. The method still converges (Figure \ref{fig:conv}, right), with order one in the uniform norm, independently of
$\beta$, of $m$ and of the spline degree: increasing $n$ reduces the error constant but not the
order. In the norm $\Cnorm{\cdot}$ the order is again one for $m\ge3$, while for $m=2$ it is lower
and compatible with $1-\beta/2$, the H\"older exponent of $u^{\star\prime}$ at the endpoints. For
$\beta=0.05$ the singularity is very weak and the orders have not yet stabilized on the meshes
considered. Graded meshes (Remark \ref{rem:saturation}) are the natural remedy and are left to
future work.

\begin{table}[H]
\centering
\begin{tabular}{c cc c cc}
\toprule
 & \multicolumn{2}{c}{growth of $\|u_{h,n}^{(m)}\|_\infty$} & order in $\|\cdot\|_\infty$ & \multicolumn{2}{c}{order in $\Cnorm{\cdot}$} \\
\cmidrule(lr){2-3}\cmidrule(lr){4-4}\cmidrule(lr){5-6}
$\beta$ & observed & $2^{\beta/2}$ & all $m$ & $m=2$ & $m=3,4$ \\
\midrule
0.75 & 1.28--1.30 & 1.297 & 0.98--1.01 & 0.62--0.72 & 0.98--1.02 \\
0.50 & 1.19       & 1.189 & 0.97--1.03 & 0.71--0.77 & 0.96--1.01 \\
0.25 & 1.09       & 1.091 & 0.91--1.15 & 0.80--0.83 & 0.93--1.02 \\
0.05 & 1.02       & 1.017 & --         & 0.93--1.00 & --         \\
\bottomrule
\end{tabular}
\caption{Exact singular solution: ranges, over $m=2,3,4$ and $n=m,\dots,m+4$, of the growth factor
of $\|u_{h,n}^{(m)}\|_\infty$ and of the estimated orders between $h=1/32$ and $h=1/64$. A dash
means that the orders have not stabilized.}
\label{tab:singularmu}
\end{table}

\section{Conclusions}
\label{sec5}

In this paper we studied a spline collocation method for nonlinear fractional boundary value problems involving the Riesz--Caputo derivative of arbitrary non-integer order $\alpha \in (m-1,m)$, $m \in \mathbb{N}$, with general two-point boundary conditions of orders $0,\dots,m-1$. We derived a new compact integral representation of the solution (Lemma \ref{lemma12}), based on the reduction of the two-sided fractional equation to a dominant singular integral equation and on its closed-form inversion (Lemma \ref{lem:inversion}, numerically validated in Section \ref{subsec41}). Building on this representation, we proved existence and uniqueness of the solution via the Banach fixed point theorem, for a class of right-hand sides which includes those generated by smooth solutions (Corollary \ref{cor:shift}), and we established, for $\alpha>1$, the convergence of the B-spline collocation method at the Greville abscissae by a residual-based stability-plus-consistency argument, with convergence order $m-\alpha$ in the norm of $C^1[0,L]$, under a uniform boundedness hypothesis on the discrete solutions which restricts the result to solutions with bounded $m$-th derivative (Remark \ref{rem:Hscope}).

The closed-form inversion also yields information that the equation alone does not provide. The
characteristic exponent of the dominant singular equation is $\beta/2$, and the corresponding
canonical function indicates the typical behaviour of $u^{(m)}$ at the endpoints, namely the weight
$\bigl(x(L-x)\bigr)^{-\beta/2}$, with an exponent depending only on the fractional order (Remark
\ref{rem:endpoint}). The weight itself, being the equilibrium density of the Riesz kernel on the
interval, solves the equation with a constant right-hand side and a constant known in closed form
(Proposition \ref{prop:exactpair}, see also \cite{Samko}); this supplies an exact test solution with the genuine endpoint behaviour,
which manufactured solutions necessarily suppress. The experiments of Section \ref{subsec42} show
that in that regime, on a uniform mesh, the attainable order is $(1-\beta)/2$ in $L^2$ and
$1-\beta/2$ in $L^1$ regardless of the spline degree --- a saturation governed entirely by the
endpoint singularity, and which meshes graded with exponent matched to $\beta/2$ are expected to
remove.

Two directions remain open. The first is the analysis on graded meshes just mentioned, for which
Remark \ref{rem:endpoint} suggests the exponent. The second is the sharpness of the
convergence estimate: as observed in Remark \ref{rem:notsharp}, the residual-based argument
bounds the residual in a fixed norm and therefore cannot see the additional regularization that
the Riesz potential provides on oscillating data. Recovering it requires a reformulation of the
scheme in which the discrete inverse acts on the discrete residual directly, and will be
addressed in a forthcoming work.

\appendix

\section{Derivation of the recursion for the boundary constants}
\label{appA}

We detail the derivation of \eqref{eq:Crec}. Write $u = V + P$ with $P(x) = \sum_{j=0}^{m-1} C_j x^j$ and $V=V(\cdot;u)$ as in \eqref{eq:Vdef}. For $0 \le r \le m-1$,
\[
u^{(r)}(x) = V^{(r)}(x) + \sum_{j=r}^{m-1} C_j\, \frac{j!}{(j-r)!}\, x^{\,j-r},
\]
so that, using $V^{(r)}(0)=0$ (see \eqref{eq:Vr0}),
\[
u^{(r)}(0) = r!\, C_r, \qquad
u^{(r)}(L) = V^{(r)}(L) + \sum_{j=r}^{m-1} C_j\, \frac{j!}{(j-r)!}\, L^{\,j-r}.
\]
Substituting into the boundary conditions $a_r u^{(r)}(0) + b_r u^{(r)}(L) = c_r$ gives, for each $r$,
\[
(a_r + b_r)\, r!\, C_r \;+\; b_r \sum_{j=r+1}^{m-1} C_j\, \frac{j!}{(j-r)!}\, L^{\,j-r} \;=\; c_r - b_r\, V^{(r)}(L),
\]
which is an upper triangular linear system in the unknowns $C_{m-1},\dots,C_0$ with nonzero diagonal entries $r!\,(a_r+b_r)$. Solving from $r=m-1$ downwards yields \eqref{eq:Crec}. Note that the system is solvable for any choice of the data, without compatibility conditions, precisely because $a_r+b_r \neq 0$ for all $r$.

For $m=2$, $a_0=a_1=b_0=b_1=1$, the recursion gives
\[
C_1 = \frac{c_1 - V'(L)}{2}, \qquad
C_0 = \frac{c_0 - V(L) - C_1 L}{2} = \frac{c_0 - V(L)}{2} - \frac{L\,(c_1 - V'(L))}{4},
\]
in agreement with a direct computation for the second-order case; no auxiliary boundary integrals are needed.

\section*{Acknowledgements}
The authors are members of the Gruppo Nazionale per il Calcolo Scientifico
(GNCS) of the Istituto Nazionale di Alta Matematica (INdAM).






\begin{thebibliography}{30}

\bibitem{BDST16}
D. Baleanu, K. Diethelm, E. Scalas, J.J. Trujillo, Fractional Calculus: Models and Numerical Methods, World Scientific, Singapore, 2016.

\bibitem{Hi00}
R. Hilfer (Ed.), Applications of Fractional Calculus in Physics, World Scientific, Singapore, 2000.

\bibitem{Ma06}
R.L. Magin, Fractional Calculus in Bioengineering, Begell House, Danbury, CT, 2006.

\bibitem{Ma10}
F. Mainardi, Fractional Calculus and Waves in Linear Viscoelasticity: An Introduction to Mathematical Models, World Scientific, Singapore, 2010.

\bibitem{Za02}
G.M. Zaslavsky, Chaos, fractional kinetics, and anomalous transport, Phys. Rep. 371 (2002) 461--580.

\bibitem{MK00}
R. Metzler, J. Klafter, The random walk's guide to anomalous diffusion: a fractional dynamics approach, Phys. Rep. 339 (2000) 1--77.

\bibitem{Pit:22}
F. Pitolli, C. Sorgentone, E. Pellegrino, Approximation of the Riesz--Caputo derivative by cubic splines, Algorithms 15 (2022) 69.

\bibitem{Sor15}
C. Sorgentone, E. Pellegrino, F. Pitolli, A spline-based framework for solving the space--time fractional convection--diffusion problem, Appl. Math. Lett. 161 (2025).

\bibitem{Di10}
K. Diethelm, The Analysis of Fractional Differential Equations: An Application-Oriented Exposition Using Differential Operators of Caputo Type, Lecture Notes in Math., vol. 2004, Springer, Berlin, 2010.

\bibitem{KST06}
A.A. Kilbas, H.M. Srivastava, J.J. Trujillo, Theory and Applications of Fractional Differential Equations, North-Holland Math. Stud., vol. 204, Elsevier, Amsterdam, 2006.

\bibitem{Pedas2011}
A. Pedas, E. Tamme, On the convergence of spline collocation methods for solving fractional differential equations, J. Comput. Appl. Math. 235 (2011) 3502--3514.

\bibitem{Pedas2011_2}
A. Pedas, E. Tamme, Spline collocation methods for linear multi-term fractional differential equations, J. Comput. Appl. Math. 236 (2011) 167--176.

\bibitem{Pedas2012}
A. Pedas, E. Tamme, Piecewise polynomial collocation for linear boundary value problems of fractional differential equations, J. Comput. Appl. Math. 236 (2012) 3349--3359.

\bibitem{Pedas2014}
A. Pedas, E. Tamme, Numerical solution of nonlinear fractional differential equations by spline collocation methods, J. Comput. Appl. Math. 255 (2014) 216--230.

\bibitem{Vainikko93}
G. Vainikko, Multidimensional Weakly Singular Integral Equations, Lecture Notes in Math., vol. 1549, Springer, Berlin, 1993.

\bibitem{Brunner}
H. Brunner, A. Pedas, G. Vainikko, Piecewise polynomial collocation methods for linear Volterra integro-differential equations with weakly singular kernels, SIAM J. Numer. Anal. 39 (2001) 957--982.

\bibitem{Pitolli_nl1}
F. Pitolli, A collocation method for the numerical solution of nonlinear fractional dynamical systems, Algorithms 12 (2019).

\bibitem{Pitolli_nl2}
E. Pellegrino, F. Pitolli, Applications of optimal spline approximations for the solution of nonlinear time-fractional initial value problems, Axioms 10 (2021) 249.

\bibitem{Samko}
S.G. Samko, A.A. Kilbas, O.I. Marichev, Fractional Integrals and Derivatives: Theory and Applications, Gordon and Breach, Yverdon, 1993.

\bibitem{Gakhov}
F.D. Gakhov, Boundary Value Problems, Pergamon Press, Oxford, 1966.

\bibitem{Tricomi}
F.G. Tricomi, Integral Equations, Interscience, New York, 1957.

\bibitem{Pi18-b}
F. Pitolli, Optimal B-spline bases for the numerical solution of fractional differential problems, Axioms 7 (2018) 46.

\bibitem{PPP:20}
E. Pellegrino, L. Pezza, F. Pitolli, Quasi-interpolant operators and the solution of fractional differential problems, in: Approximation Theory XVI, Nashville, 2019, 2020.

\bibitem{Ag07}
O.P. Agrawal, Fractional variational calculus in terms of Riesz fractional derivatives, J. Phys. A: Math. Theor. 40 (2007) 6287.

\bibitem{Po99}
I. Podlubny, Fractional Differential Equations, Math. Sci. Eng., vol. 198, Academic Press, San Diego, 1999.

\bibitem{Rieszcos}
M. Partohaghighi, E. Asante-Asamani, O. Iyiola, A robust numerical scheme for solving Riesz-tempered fractional reaction--diffusion equations, J. Comput. Appl. Math. 450 (2024).

\bibitem{Landkof}
N.S. Landkof, Foundations of Modern Potential Theory, Springer, Berlin, 1972.

\bibitem{dB78}
C. de Boor, A Practical Guide to Splines, Springer, New York, 1978.

\bibitem{Sc07}
L.L. Schumaker, Spline Functions: Basic Theory, third ed., Cambridge University Press, Cambridge, 2007.

\bibitem{Sb05}
P. Sablonni\`ere, Recent progress on univariate and multivariate polynomial and spline quasi-interpolants, in: M.G. de Bruin, D.H. Mache, J. Szabados (Eds.), Trends and Applications in Constructive Approximation, vol. 177, Birkh\"auser, Basel, 2005, pp. 229--245.

\end{thebibliography}
\end{document}